\documentclass[11pt]{article}
\usepackage{latexsym}
\usepackage{amssymb}
\usepackage{amsfonts}
\usepackage{amsmath}
\usepackage{cancel}
\usepackage{xcolor}
\usepackage{hyperref}

\definecolor{bleu}{rgb}{0.00,0.4,0.90}

\usepackage[pagewise]{lineno}
\makeatletter
\long\def\unmarkedfootnote#1{{\long\def\@makefntext##1{##1}\footnotetext{#1}}}
\makeatother

\newtheorem{definition}{Definition}[section]

\newtheorem{lemma}[definition]{Lemma}

\newtheorem{theorem}[definition]{Theorem}

\newtheorem{proposition}[definition]{Proposition}

\newtheorem{corollary}[definition]{Corollary}

\newtheorem{remark}[definition]{Remark}

\newtheorem{example}[definition]{Example}

\newcommand{\Om}{\Omega}

\newcommand{\R}{\mathbb R}

\newcommand{\dive}{\mathrm{div}\;}

\newcommand{\de}{\partial}
\newcommand{\RN}{{\mathbb R}^{N}}
\newcommand{\Si}{\sum_{i=1}^N}

\newcommand{\dei}{\partial_{x_i}}

\newenvironment{proof}[1][Proof]{\textbf{#1.} }{\ \rule{0.5em}{0.5em}}
\begin{document}

\title{Regularity results for local solutions \\ to some anisotropic elliptic equations}

\author{ G. di
Blasio\thanks{Dipartimento di Matematica e Fisica,
Universit\`{a} degli Studi della Campania \textquotedblleft L. Vanvitelli\textquotedblright , Viale Lincoln, 5 - 81100
Caserta, Italy. E--mail: giuseppina.diblasio@unicampania.it} -- F.
Feo\thanks{Dipartimento di Ingegneria, Universit\`{a} degli Studi di
Napoli \textquotedblleft Parthenope\textquotedblright, Centro
Direzionale Isola C4, 80143 Napoli, Italy. E--mail:
filomena.feo@uniparthenope.it}
-- G. Zecca \thanks{
Dipartimento di Matematica e Applicazioni \textquotedblleft R. Caccioppoli\textquotedblright,
 Universit\`{a} degli Studi di Napoli Federico II, Complesso Universitario di Monte S. Angelo, Via Cintia,  80126, Napoli, Italia. E--mail: g.zecca@unina.it }}

\date{}

\maketitle
\begin{abstract}

In this paper we study the higher integrability of local solutions for a class of anisotropic equations with lower order terms whose growth coefficients lay in Marcinkiewicz spaces.
A condition for the boundedness of such solutions is also given.

\end{abstract}

\bigskip




\bigskip

\numberwithin{equation}{section}

\section{Introduction}

Our aim is to obtain regularity results for the following class of anisotropic elliptic equations
\begin{equation}\label{diffusion}
 \sum_{i=1}^N \de_{x_i}  (\mathcal{A}_i(x,\nabla u)+\mathcal B_i(x,u)) =    \sum_{i=1}^N {\de_{x_i}} ( |\mathcal F_i|^{p_i-2} \mathcal F_i) \qquad \mbox { in } \Omega,
\end{equation}
where $\Omega $  is a domain of $ \RN$, $N>2$, $p_i>1$ for every $i=1,...,N$ with
$1< \bar p<N $, denoting by $\overline{p}$ the harmonic mean of
$p_1,\cdots,p_N$, \textit{i.e.}
\begin{equation}\label{p bar}
\frac{1}{\overline{p}}=\frac{1}{N}\sum_{i=1}^N\frac{1}{p_i}.
\end{equation}

Throughout this paper, we make the following assumptions for any $i=1,...,N$
\bigskip

\noindent
$ (\mathcal H 1)\quad \mathcal A_i: \Om \times \mathbb R^N \rightarrow \R $ is a Carath\'eodory function
that satisfies
\begin{eqnarray}
\label{ii1}&&| \mathcal A_i (x, \xi ) | \leqslant \beta_i  | \xi_i|^{p_i-1}\\
\label{ii2}&&  \alpha \sum_{i=1}^N |\xi_i|^{p_i}\leqslant \sum_{i=1}^N  ( \mathcal A_i(x, \xi )\, \xi_i)
\end{eqnarray}

\noindent for a.e. $x \in \Om$ and for any vector $\xi$  in $\mathbb R^N$, where $0<\alpha\leqslant \beta_i$ are constants;
\bigskip

\noindent
$(\mathcal H 2) \quad \mathcal B_i:\Om \times \mathbb R \rightarrow \mathbb R $ is a Carath\'eodory  function such that
\begin{equation}\label{b1}
|\mathcal B_i(x,s)| \leqslant b_i(x)|s|^{ \frac{\bar{p} }{p_i'}}
\end{equation}
for a.e. $x\in \Om$ and for every $s\in \mathbb R$, where $b_i$ is
 a non negative function in the  {Marcinkiewicz space} $ L^{\frac{N\,p'_i} {\bar p},\infty}(\Om)$, {{where $p'=\frac{p}{p-1}$}};

\bigskip

\noindent
$(\mathcal H 3)\quad \mathcal F_i:\Om\to \R$  is such that $ \mathcal F_i\in L_{loc}^{ p_i}(\Omega )$.

\bigskip

%

Our model operator is
\begin{equation*}\label{operatorL}
L (u) = \sum_{i=1}^N \de_{x_i} \left(
|\de_{x_i}  u|^{p_i-2}\de_{x_i}  u
+ b_i(x)|u|^{\frac{\bar{p}}{p'_i} -1}u
\right)
,
\end{equation*}
with $b_i$ as in $(\mathcal H 2)$. The anisotropy of  $L$  is due to different power growths with respect to the partial derivatives of the unknown $u$
and it coincides with  the  so-called pseudo-Laplacian operator when  $b_i=0$ and $p_i=p$ for $i=1,....,N $.

Let us point out that term anisotropy is used in various scientific disciplines and could have a different meaning when it is related to equations as well.  The interest in anisotropic problems has deeply increased in the last years and many results in different directions have been obtained.  We quote a list of references that is obviously not exhaustive and we refer the reader to references therein to extend it: \cite{AdBF1,AdBF2,AC,BK,BGM,BB,BaCri,C1,C2,CMM,Dic1,Fragala,FGL,Mosc,LI}.

Several regularity results depending on the summability of datum are well-known in
literature for isotropic counterpart of \eqref{diffusion}
\begin{equation*}\label{diffusion2}
 \dive \left(|\nabla u|^{p-2} \nabla u +b(x)|u|^{p-2}u\right) = \dive    ( |\mathcal F|^{p-2} \mathcal F) \qquad \mbox { in } \Omega,
\end{equation*}
where $p>1$, $b$ and $\mathcal F$ are vector fields with suitable summability.
In the isotropic case for linear equations, assuming the coefficient of the lower order term in suitable Lebesgue spaces, Stampacchia in \cite{S} proves that if the datum $\mathcal F$ belongs to $(L^r(\Omega))^N$ with $2<r<N$, then the solution $u$ belongs to $L^{r^*}(\Omega)$. Otherwise when $r>N$ it follows that the solution $u$ is bounded. Similar results have been obtained also for isotropic nonlinear operators taking the coefficient of the lower order term $b$ in Lebesgue spaces in \cite{B}
and in the Marcinkiewicz spaces in \cite{Bocc1}, \cite{FGMZ}, \cite{GGM} and \cite{GMZ}. In this paper we prove such kind of regularity results for anisotropic equation \eqref{diffusion} dealing with local solution, whose definition is recalled in what follows.

\begin{definition}
If $\mathcal F_i\in L_{loc}^{ p_i}(\Omega )$ for $i=1,..,N$, we say that
 $u\in W^{1,\vec {p}}_{loc} (\Omega )$ is a local solution to  \eqref{diffusion} provided
 \begin{equation}\label{sol1}
\Si \int_\Om   \left (\mathcal{A}_i(x,\nabla u) +\mathcal B_i(x,u) \right) \dei \varphi \, dx = \Si \int_\Om  ( \mathcal |F_i|^{p_i-2} \mathcal F_i)\, \dei \varphi \,dx
\end{equation}
$\forall \varphi\in C^\infty_0(\Om)$.
\end{definition}

For the definition  of the  anisotropic Sobolev space $W^{1,\vec {p}}_{loc} (\Omega )$  we refer to Section \ref{anisotropicSec}.

 To give an idea  of our result,  let us consider, for simplicity, equation  \eqref{diffusion} without the lower order terms, \emph{i.e.} $\mathcal B_i\equiv0$. When  the datum $\mathcal{F}_i\in L_{loc}^{p_i}(\Omega)$ for $i=1,..,N$ with $\bar p<N$, a local solution $u$ belongs to $L_{loc}^{\overline{p}^*}(\Omega)$, where
\begin{equation}\label{p star}
\bar{p}^*=\frac{N\bar{p}}{N-\bar p},
\end{equation}
as suggested by the anisotropic imbedding (see Section 4). Otherwise if $\mathcal{F}_i\in L_{loc}^{r_i}(\Omega)$ with $r_i>p_i$ for $i=1,..,N$, we expect that the summability of $u$ improves. In order to analyze the higher summability of $u$ we consider the following minimum
\begin{equation}\label{mu}
\mu=\min_{i}\left\{ \frac{r_i}{p_i}\right\},
\end{equation}
first introduced in \cite{BMS},  and we are able to prove that $u\in L_{loc}^{(\overline{p}\mu)^*}(\Omega)$ if $\bar p\mu<N$. Then the regularity of $u$ depends on $\mu$.

These regularity results are stated in Theorem \ref{mainTHM} taking into account the lower order terms under the assumption that coefficients $b_i$  have a suitable distance to $L^\infty$ sufficiently small for $i=1,\cdots,N$.

The principal difficulties  are due to the anisotropy, to the managing of local solutions and to the presence of the lower order terms. In our proof, one of the key tool is a new anisotropic Sobolev inequality in Lorentz spaces that involves the product of different powers of two functions (see Proposition \ref{Prop dis}). This inequality naturally appears in the anisotropic framework, it is of independent interest and gives an estimate in terms of the norm of the geometric mean of the partial derivatives instead of the geometric mean of the norms of the partial derivatives as usual in literature (see \eqref{Ex1}).

 Moreover in Theorem \ref{Th2} we give a sufficient condition in terms of $\mu$ for the boundedness of the solutions to \eqref{diffusion}.

We also make comments on the regularity of weak solutions of  Dirichlet problems in a bounded open set $\Omega \subset \mathbb{R}^N$ with Lipschitz boundary. In this case when $\mathcal{F}_i\in L^{p_i}(\Omega)$ for $i=1,..,N$ with $\bar p<N$ a weak solution $u$ belongs to  $L^{p_{\infty}}(\Omega)$, where $p_{\infty}=\max\{\overline{p}^*,p_{\max}\}$ with $p_{\max}=\max\{p_1,\cdots,p_N\}$ and $\bar p$ defined as in \eqref{p star}. It is evident that the regularity of $u$ depends on how much the anisotropy is concentrated, so the situation is more diversified than the isotropic case.  Otherwise if $\mathcal{F}_i\in L^{r_i}(\Omega)$ with $r_i>p_i$ for $i=1,..,N$ and $\mu\bar p<N$, we get $u\in L^s(\Omega)$ with $s=\max\{(\overline{p}\mu)^*,\mu p_{\max}\}$ and the regularity again depends on how much $p_i$ are spread out and on $\mu$.

\medskip

The paper is organized as follows. The main results are stated in Section 2. In Section 3 we recall same properties of Lorentz spaces and in Section 4 we introduce the anisotropic Sobolev spaces and the related inequalities. The proofs of  main results (Theorem \ref{mainTHM} and Theorem \ref{Th2}) are given in Section 5 and 6. We conclude the paper with a technical lemma contained in the  Appendix.

\section{Main results}

The first  result of this paper states the regularity of local solutions to \eqref{diffusion} in terms of  the summability of $\mathcal F_i$, replacing the usual smallness assumption on the norm of the coefficients of the lower order terms with a weaker one given in terms of the distance of a function $f\in L^{p,\infty}(\Om)$ to $  L^\infty(\Om)$, denoted by ${\rm{dist}}_{L^{p,\infty} (\Om) }(f,L^\infty(\Om))$ and defined by \eqref{dist_infty} in Section 3.

\begin{theorem}\label{mainTHM}
Assume that  \eqref{ii1}-\eqref{b1} are fulfilled. Let $1<\bar p<N$, $\bar p^*> p_{\max}$ and let $r_1,\cdots,r_N$ be such that
\begin{equation}\label{hp-ri}
1<\mu< \frac{N}{\bar p},
\end{equation}
 where $\mu$ is defined in \eqref{mu}.
 There exists a positive constant  $d=d(\vec r,N, \alpha, \vec{p})$  such that if
\begin{equation}\label{hpdistmain}
\max_i\left\{{ \emph{dist}}_{L^{\frac{N p'_i }{\bar p},\infty}(\Om)} (b_i, L^\infty(\Om))\right\} <d
\end{equation}
and  $u\in W^{1,\vec p}_{loc}(\Om)$
is a local solution to \eqref{diffusion} with $\mathcal F_i\in L_{loc}^{r_i} (\Om)$ for $i=1,..,N$, then
$|u|^{\frac{N\bar p(\mu-1)}{p_i(N-\bar p\mu)}+1}\in W^{1,\vec {p}}_{loc} (\Omega )$.
In particular $u\in L^{s}_{loc}(\Om)$, where
\begin{equation}\label{s}
s=(\mu\bar p)^*= \frac{N\mu \bar p}{N-\mu\bar p}.
\end{equation}
\end{theorem}

Note that if $\mu$ goes to $\left(\frac{N}{\overline{p}}\right)^+$, then $s\rightarrow\infty$ as expected.
In the isotropic case condition $\bar p^*>p_{\max}$  does not turn up and
assumption \eqref{hp-ri} reads as isotropic assumption $r<N$ considering the datum $\mathcal F=(\mathcal F_1,\cdots,\mathcal F_N) \in (L^r_{loc}(\Omega))^N$ (see \cite{GMZ} and \cite{FGMZ}).

 Some comments on the case $p_{\max}\geq \overline{p}^*$ are contained in Remark \ref{remark0},  Remark \ref{remark1} and Remark \ref{remarkFINE}, where weak solutions for Dirichlet problems are taking into account.


It is clear that in Theorem \ref{mainTHM} regularity of $u$ is related to $\mu p_1,\cdots,\mu p_N$. Indeed if $\vec {r}$ and $\vec {q}$  are two different vectors such that  $\min_{i}\left\{ \frac{r_i}{p_i}\right\}=\min_{i}\left\{ \frac{q_i}{p_i}\right\}=\mu$, then the solution $u\in L^{(\mu \bar p)^*}_{loc}(\Om)$  if either $\mathcal F_i\in L_{loc}^{r_i} (\Om)$ or $\mathcal F_i\in L_{loc}^{q_i} (\Om)$, as in the case $\mathcal{F}_i\in L^{\mu p_i}_{loc}(\Omega)$ for $i=1,..,N$.


A standard approach to treat the presence of lower order terms is to require a smallness on the norm of $b_i$, which is avoid using assumption \eqref{hpdistmain}, firstly introduced by \cite{GGM} in the isotropic case.
  We stress that the value of $d$ in  \eqref{hpdistmain} really depends on $\vec r$ (see Example 3.1 in \cite{GMZ}).

 For example if $\Omega$ is the ball centered at the origin with radius $R>0$,   we can take $b_i(x)=\gamma_i|x|^{- \frac{\bar{p} }{p_i'}}+h_i(x)$ with $\gamma_i>0$ and $h_i\in L^\infty(\Om)$. It is not difficult to see that $b_i\in L^{\frac{N\,p'_i} {\bar p},\infty}(\Omega)$ and verifies \eqref{hpdistmain} for suitable $\gamma_i$. We emphasize that condition \eqref{hpdistmain} is trivially satisfied whenever $b_i$ belongs to a Lebesgue space or to any Lorentz space contained in $ {L^{\frac{N p'_i }{\bar p},\infty}(\Om)}$, $i=1,...,N$. Then as a corollary of Theorem \ref{mainTHM} we have immediately the following result.


\begin{corollary}
Assume that \eqref{ii1}-\eqref{b1} are fulfilled. Let  $1<\bar p<N$, $\bar p^*> p_{\max}$, $1<q<\infty$,
\begin{equation*}
b_i\in L^{\frac{N p'_i }{\bar p},q}(\Om) \,\,\qquad   i=1,..,N
\end{equation*}
and let $r_1,\cdots,r_N$ be such that \eqref{hp-ri} holds.
If  $\mathcal F_i\in L_{loc}^{r_i} (\Om)$ for $i=1,..,N$,
 then any local solution   $u\in W^{1,\vec p}_{loc}(\Om)$
 to \eqref{diffusion} belongs to $L^{s}_{loc}(\Om),$
where $s$ is defined in \eqref{s}.
\end{corollary}



   Without lower order terms in \cite{BMS} the authors have proved that the boundedness of a weak solution of Dirichlet problems is guaranteed under the assumption
\begin{equation}\label{hp-ri_boun}
\mu> \frac{N}{\bar p},
\end{equation}
where $\mu$ is defined in \eqref{mu}. However if $\mathcal{B}_i\not\equiv0$  for $i=1,\cdots,N$ the boundedness is not assured assuming that \eqref{hpdistmain} is in force as showed in Example 4.8 of \cite{GGM} (when $p_i=2$  for $i=1,\cdots,N$). The smallness of $\|b_i\|_{L^{\frac{Np_i'}{\bar p},\infty}}$ for $i=1,\cdots,N$ neither is sufficient to get boundedness, as the following example shows.

\begin{example}
Let us consider the following Dirichlet problem
\begin{equation*}
\left\{
\begin{array}
[c]{ll}%
 \emph{ \text{div}} (\nabla u+   bu) = \emph{\rm \text{div}}
\mathbf{\mathcal{F}} &
\hbox{ in } \Omega
\\
& \\
u=0  & \hbox{ on }\partial\Omega,
\end{array}
\right.
\label{PPbis}%
\end{equation*}
where $N>2$, $\Omega=B(0,1)$ is the unit ball of $\mathbb R^N$, $b(x)=\frac{\gamma x}{|x|^2}$, $\mathbf{\mathcal{F}}(x)=\left((2+\gamma)x_1,\cdots, (2+\gamma)x_N\right)$ and  $\gamma>0$. We stress that the solution $u(x)=|x|^{-\gamma}-|x|^2$ is unbounded even if $\|b\|_{L^{N,\infty}}=\gamma\omega_N^{1/N}$ can be small as we want taking $\gamma$ small enough.

\end{example}
In order to  obtain an $L^\infty-$regularity result we require extra summability on the  coefficients of lower order terms.
\begin{theorem}\label{Th2}
Assume that  \eqref{ii1}-\eqref{b1} holds. Let $1<\bar p<N$, $\bar p^*> p_{\max}$, $b_i\in L^{\frac{r_i}{p_i-1}}_{loc}(\Om)$
and  $\mathcal F_i\in L^{r_i}_{loc} (\Om)$ for $i=1,..,N$ with $r_1,\cdots,r_N$ satisfying \eqref{hp-ri_boun}. Then any local solution $u\in W^{1,\vec p}_{loc}(\Om)$ to equation \eqref{diffusion}
is locally bounded.
\end{theorem}

\section{Some properties of Lorentz spaces}\label{section}

In this section we recall the definitions of Lorentz spaces and their properties (see \cite{PKOS}  for more details).
There are various definitions of Lorentz spaces but all of them manage the notion of rearrangement.

Here we  assume that $\Om\subset \RN$, $N>2$ is an open set. Let $v$ be a measurable function defined in $\Om.$ The distribution function $\mu_{v}:[0, +\infty) \rightarrow [0, +\infty)$ of $v$ is defined as
\[
\mu_{v}(\lambda):=\big|\left\{ x\in {\Om}: |v(x)|> \lambda  \right\}\big| \qquad \text{for } \lambda\geqslant 0.
\]
The {decreasing rearrangement} of $v$ is the map $v^{*}: [0, +\infty) \rightarrow [0, +\infty]$ given by
\begin{equation}\label{u star}
v^{*}(s)= \sup \left\{   t\geqslant 0: \mu_v(t) \geqslant s\right\}\qquad \text{ for } s\geqslant 0.
\end{equation}
By $v^{**}$ we denote the {\em maximal function} of $v^{*}$, \emph{i.e.}
\begin{equation}\label{u star star}
v^{**}(s)= \frac{1}{s} \int_0^s v^{*}(\sigma)\, d\sigma\qquad \text{ for } s>0.
\end{equation}

For {$1\leqslant p<+\infty$} and $0<q\leqslant + \infty$ the Lorentz space $L^{p,q}({\Om})$ consists in all measurable functions $v:\Om\to \R$ such that
\begin{equation*}
\|v\|_{L^{p,q}({\Om})}=
\begin{cases}
\displaystyle{\bigg(\int_0^{|{\Om}|} \big[t^{1/p}\,v^*(t)\big]^q\,\frac{dt}t\bigg)^{1/q}}\qquad&\text{if $q<+\infty$,}\\[5mm]
\displaystyle{\sup_{t\in(0,|{\Om}|)}\big[t^{1/p}\,v^*(t)\big]}&\text{if $q=+\infty$}
\end{cases}
\end{equation*}
 is finite.  Notice that in general $\|\cdot\|_{L^{p,q}({\Om})}$ is a  quasinorm, but replacing $v^*$ with $v^{**}$ one obtains an equivalent norm and $L^{p,q}(\Omega)$ becomes a Banach space. We observe that it holds
\begin{equation*}\label{normepq}
\| \,|v|^s\|_{L^{p,q}}=\|v\|_{L^{s p,s q}}^{s } \quad \text{ for } 0<s<+\infty.
\end{equation*}

The Lorentz spaces are a refinement of Lebesgue spaces. Indeed for $p=q$,
$L^{p,p}(\Omega)$ reduces to the standard Lebesgue space $L^p(\Omega)$,  and  $L^{p,\infty}(\Om)$ is also known as Marcinkiewicz space $\mathcal M^p({\Om})$, or weak-$L^p({\Om})$. If $\Om$ is bounded, such inclusions follow
\[
L^r (\Om) \subset   L^{p,q}(\Om)\subset L^{p,p}(\Om)\equiv L^p(\Om)\subset  L^{p,r} (\Om) \subset L^{p,\infty}(\Om)\subset  L^q(\Om)
\]
for $1<q<p<r\leqslant +\infty,$ with continuous injections. We observe that $L^{p,\infty}(\Omega) \supset L^p(\Omega)$. For example, if $\Omega \subset \mathbb R^N$ contains the origin, the function
$v(x)=\vert x \vert^{-\frac{N}{p}}\not\in L^p(\Omega)$ but $v\in L^{p, \infty}(\Omega)$ with $\vert \vert v \vert \vert_{L^{p,\infty}}^p=\omega_N$, where  $\omega_N$ stands for the Lebesgue measure of the unit ball of $\mathbb R^N.$

Now we introduce a suitable characterization of Lorentz spaces that is useful to generalize Sobolev embedding theorems in anisotropic framework.

Let $k>1$ fixed, for any measurable function $v$ defined in $\Omega$ such that for every $\lambda>0$
\begin{equation}\label{level_v}
\mu_v (\lambda)<+\infty,
\end{equation}
we choose the levels $a^v_n\geq 0 , n\in \mathbb{Z}$, such that

\[
|\{x\in \Omega: |v(x)|> a^v_n\}|\leq k^{-n}\leq |\{x\in \Omega: |v(x)|\geq a^v_n\}|
\]
or equivalently
\begin{equation}\label{def_an}
a^v_n\in [v^*((k^{-n})^+);v^*((k^{-n})^-)]
\end{equation}
where $v^*$ is the decreasing rearrangement of $v$ defined in \eqref{u star} and where plus and minus denote the right and the left-hand limit respectively.

Now we consider a particular sequence of functions $\omega_n(v)$  defined by
 \begin{equation}
\omega_n(v)=\left\{
\begin{array}
[c]{ll}%
0 & \hbox{ for }0\leq|v|\leq a^v_{n-1}\\
|v|-a^v_{n-1} & \hbox{ for }a^v_{n-1}<|v|\leq a^v_{n} \\
a^v_{n}-a^v_{n-1} & \hbox{ for }|v|> a^v_{n} ,
\end{array}
\right.  \label{omega n}%
\end{equation}
with the levels $a^v_n$ defined for $n\in \mathbb{Z}$ as in \eqref{def_an}. We observe that
\begin{equation}\label{dis carat}
\omega_n(v)\leq (a^v_{n}-a^v_{n-1})\chi_{\{|v|>a^v_{n-1}\}}\quad \text{and} \quad\omega_n(v)\geq (a^v_{n}-a^v_{n-1})\chi_{\{|v|\geq a^v_{n}\}},
\end{equation}
where $\chi$ is the characteristic function, so one finds for $0<r<\infty$, that
\[
k^{-\frac{n}{r}}(a^v_{n}-a^v_{n-1})\leq \left(\int_{\mathbb{R}^N}|\omega_n(v)|^r dx\right)^{\frac{1}{r}}\leq k^{-\frac{n-1}{r}}(a^v_{n}-a^v_{n-1}).
\]

\noindent Recalling that  a sequence $ \{a_n\}_{n\in \mathbb Z}$ belongs to $l^q(\mathbb Z)$ for $q>1$ iff $
\sum_{n\in \mathbb Z} |a_n|^q<+\infty$,
we are in position to state the following announced  characterization of Lorentz spaces.

\begin{proposition}\label{Th_equiv}
Let be $1<p<\infty$ and $1\leq q\leq \infty$. For any $v$ extended by $0$ outside $\Omega$ satisfying \eqref{level_v}, one has
\begin{equation}\label{equiv}
  v\in L^{p,q}(\mathbb{R}^N) \quad \text{is equivalent to} \quad k^{-\frac{n}{p}}(a^v_{n+1}-a^v_{n}) \in l^q(\mathbb Z)
\end{equation}
and in particular
\begin{equation}\label{equiv_bis}
\!\!\!\!\left\|v\right\|_{L^{p,q}(\Omega)}=\left\|v\right\|_{L^{p,q}(\mathbb{R}^N)}
\leq C_1
\left(\sum_{n \in\mathbb{Z}}\left[a^v_{n}\right]^q
k^{-\frac{nq}{p}}
\right)^{1/q} \!\!\!
\leq C_2
\left(\sum_{n \in\mathbb{Z}}\left[a^v_{n}-a^v_{n-1}\right]^q
k^{-\frac{nq}{p}}
\right)^{1/q}\!\!,
\end{equation}
with $C_1,C_2$ are positive constants independent on $v$ and $a^v_{n}$ is defined in \eqref{def_an}.
\end{proposition}
\begin{proof}
Equivalence \eqref{equiv} is contained in \cite{Ta}, Proposition 3. We prove  \eqref{equiv_bis}. Since the rearrangement is non increasing, we have
\begin{align}\label{Lor_0}
\!\!\left\|v\right\|_{L^{p,q}}^{q}
& \!= \!\sum_{n \in\mathbb{Z}}\int_{k^{-n}}^{k^{-(n-1)}}\!\left[s^{1/p}v^*(t)\right]^{q} \frac{ds}{s} \leq \sum_{n \in\mathbb{Z}}k^{-\frac{(n-1)q}{p}}
[v^*((k^{-n})^{+})]^{q}\int_{k^{-n}}^{k^{-(n-1)}}\!\!\frac{ds}{s}
\\
&=\sum_{n \in\mathbb{Z}}[v^*((k^{-n})^{+})k^{-\frac{n-1}{p}}]^{q} \log k
\leq \log k \sum_{n \in\mathbb{Z}}( k^{-\frac{n-1}{p}}a^v_n)^{q}=k^{\frac{q}{p}}\log k \sum_{n \in\mathbb{Z}}( k^{-\frac{n}{p}}a^v_n)^{q}.\nonumber
\end{align}
 Now following the idea of Tartar \cite{Ta}, we put $b_n=a^v_n-a^v_{n-1}$. We observe that $a^v_n=\sum_{m=-\infty}^{n} b^v_m$ since the measure of the level sets is finite. We have that

\begin{equation*}
k^{-\frac{n}{p}}a^v_n= \sum_{m=-\infty}^{n} k^{\frac{m-n}{p}}(k^{-\frac{m}{p}}b^v_m)
\end{equation*}
and by the convolution Young inequality, we deduce  the following inequality
\begin{equation}\label{convolution_0}
\begin{split}
  \left(\sum_{n\in \mathbb{Z}}|k^{-\frac{n}{p}}a^v_n|^q\right)^{\frac{1}{q}}
  &
  \leq\left(\sum_{m\leq 0}k^{\frac{m}{p}}\right)
  \left(\sum_{n\in \mathbb{Z}}|k^{-\frac{n}{p}}(a^v_n-a^v_{n-1})|^q\right)^{\frac{1}{q}}\\
  &= k^{-\frac{1}{p}}\left(\sum_{m\leq 0}k^{\frac{m}{p}}\right)
  \left(\sum_{n\in \mathbb{Z}}|k^{-\frac{n}{p}}(a^v_{n+1}-a^v_{n})|^q\right)^{\frac{1}{q}}.
  \end{split}
\end{equation}
Combining \eqref{convolution_0} and \eqref{Lor_0}, inequality \eqref{equiv_bis} follows
with $C_2=(\log k)^{\frac{1}{q}}\sum_{m\leq 0}k^{\frac{m}{p}}$.
\end{proof}

\bigskip


\bigskip
We remark that $L^\infty(\Om)$ is not dense in $L^{p,\infty}(\Om)$,  $p\in{}]1, +\infty[$. We define the distance of a given  function $f\in L^{p,\infty}(\Om)$ to $  L^\infty(\Om)$ as
\begin{equation}\label{dist_infty}
{\rm{dist}}_{L^{p,\infty} (\Om) }(f,L^\infty(\Om))=\inf_{g\in L^\infty(\Om)} \|f-g\|_{L^{p,\infty}(\Om)}.
\end{equation}
Note that, since $\|~\|_{p,\infty}$ is not a norm, ${\rm{dist}}_{L^{p,\infty} (\Om)}$ is just equivalent to a metric. In \cite{CS} is proved that
\begin{equation}\label{distlim}
{\rm{dist}}_{ L^{p,\infty}(\Om)  }(f, L^\infty(\Om))=\lim_{M \to +\infty}\|f-T_M f\|_{L^{p,\infty}(\Om)},
\end{equation}
where the truncation at level $M>0$ is defined as
\begin{equation}\label{T}
T_M(y)=\frac{y}{|y|}\min\{|y|, M\}.
\end{equation}

At the end of this section we recall the following useful lemma.

\begin{lemma}\label{Lemma tetai}\emph{(}see \cite[page 43]{KPS}\emph{)} Let $X$ be a rearrangement invariant space and let $
0\leq \theta _{i}\leq 1$ for $i=1,...,M,$ such that $\sum_{i=1}^{M}\theta _{i}=1$, then
\begin{equation*}
\left\| \prod_{i=1}^{M}|f_{i}|^{\theta _{i}}\right\| _{X}\leq
\prod_{i=1}^{M}\Vert f_{i}\Vert _{X}^{\theta _{i}} \quad \forall f_i\in X.  \label{holder
X}
\end{equation*}
\end{lemma}


\section{Anisotropic inequalities}\label{anisotropicSec}
Let $\vec{p}= (p_1,p_2,...,p_N)$ with $p_i>1$ for $i=1,...,N$ and $\Omega$ be a bounded open set. As usual the anisotropic Sobolev space is the Banach space defined as

\[
W^{1,\vec{p}}(\Om)= \{ u\in W^{1,1} (\Om) : \de_{x_i} u\in L^{p_i} (\Om),i=1,...,N\}
\]
equipped with \[
\|u\|_{W^{1,\vec{p}}(\Om)} = \|u\|_{L^{1}(\Om)} +\sum_{i=1}^N  \|\de_{x_i} u\|_{L^{p_i}(\Om)}
\]
and
\[
W^{1,\vec{p}}_{loc}(\Om)= \{ u\in W^{1,\vec{p}}(\Om')   :  \Om'\subset\subset  \Om  \mbox{ open set}\}.
\]

It is well-known that in the anisotropic setting a Poincar\`e type
inequality holds true (see \cite{Fragala}). Indeed for every $u\in C_0^{\infty}(\Om)$ with $\Omega$ a bounded open set with Lipschitz boundary we have
\begin{equation}\label{dis poincare}
\|u\|_{L^{p_i}(\Om)} \leq C_P  \|\de_{x_i} u\|_{L^{p_i}(\Om)},\qquad \qquad i=1,...,N
\end{equation}
for a constant $C_P$ depending on the diameter of $\Omega$. Moreover, for $u\in C_0^{\infty}(\mathbb{R^N})$ the following anisotropic Sobolev inequality holds true (see\cite{Ta})
\begin{equation}\label{Sob}
\|u\|_{L^{p,q}(\mathbb{R^N})} \le S_N\prod_{i=1}^N \| \de_{x_i} u\|_{L^{p_i}(\mathbb{R^N})}^{\frac 1N},
\end{equation}
where $S_N$ is an universal constant and  $p=\bar{p}^*$
and $q=\bar{p}$ whenever $\bar{p}<N$, where $\bar p$ is defined in \eqref{p bar}.
Using the inequality between geometric and arithmetic mean we can replace the right-hand-side of \eqref{Sob} with $ \sum_{i=1}^N \| \de_{x_i} u\|_{L^{p_i} }$.

When $\overline{p}<N$ and $\Omega$ is a bounded open set with Lipschitz boundary, {the space $W_{0}^{1,\overrightarrow{p}}(\Omega)= \overline{ C_0^\infty (\Om)}^{\sum_{i=1}^N \| \de_{x_i} u\|_{L^{p_i} }}$ is continued embedding into $ L^{q}(\Omega)$ for $q\in\lbrack1,p_{\infty}]$, with
$p_{\infty}:=\max\{\overline{p}^{\ast},p_{\max}\}$ as a consequence of \eqref{Sob} and \eqref{dis poincare}}.

In the following proposition we generalize the anisotropic Sobolev inequality  \eqref{Sob} to the product of functions.
\begin{proposition}\label{Prop dis}
Let $\alpha_i,\beta_i\geq0$ (but not both identically zero) and $p_1,\cdots,p_N\geq1$ be such that $1\leq \overline{p}<N$. Then for every nonnegative functions $\phi,u\in C^\infty_0(\mathbb{R}^N)$ we have
\begin{equation}\label{Lor ineq bis zero}
\left\|\phi^{\underline{{\alpha}}}u^{\underline{\beta}}\right\|_{L^{\overline{p}^*,\overline{p}}}\leq C
\left\|\left(\prod_{i=1}^{N}\frac{\partial}{\partial x_i} (\phi^{\alpha_i}u^{\beta_i})\right)^{1/N}\right\|_{L^{\overline{p}}},
\end{equation}
where $\overline{p}^*$ is defined in \eqref{p star}
and $\underline{{\alpha}}$ and $\underline{{\beta}}$ are defined by
\begin{equation*}\label{alfa bar}
\underline{\alpha}=\frac{1}{N}\sum_{i=1}^N\alpha_i,\qquad \underline{\beta}=\frac{1}{N}\sum_{i=1}^N\beta_i
\end{equation*}
and $C$  is a positive constant independent on $u$ and $\phi$.
\end{proposition}

\noindent It is clear that for example for $\beta_i=1$ the sign assumption on $u$ can be dropped.

\noindent We point out that Lemma \ref{Lemma tetai}  with $X=L^{\overline{p}}(\Omega)$ and $\theta_i=\frac{\overline{p}}{p_iN}$ yields  that
\begin{equation}\label{Lor ineq bis}
\left\|\phi^{\underline{{\alpha}}}u^{\underline{\beta}}\right\|_{L^{\overline{p}^*,\overline{p}}}\leq C \prod_{i=1}^N\left\|\frac{\partial }{\partial x_i}(\phi^{\alpha_i}u^{\beta_i})\right\|^{1/N}_{L^{p_i}}\quad \forall \phi,u\in C^\infty_0(\mathbb{R}^N)
\end{equation}
with  a positive constant $C$ independent on $u$ and $\phi$.
We emphasize that \eqref{Lor ineq bis zero} is sharper than \eqref{Lor ineq bis} as the following example shows.
\begin{example}\label{Ex1}
Let us consider $1<\bar p<N$, $\Omega=\left\{x:\overset{N}{\underset{i=1}{{\sum}}}|x_i|^{\theta_i}<R\right\},\phi\equiv1,\alpha_i\equiv1$ $$u(x)=\left(\overset{N}{\underset{i=1}{{\sum}}}|x_i|^{\theta_i}\right)^{-\gamma}
-\left(\overset{N}{\underset{i=1}{{\sum}}}R^{\theta_i}\right)^{-\gamma}$$
with $R>0,\theta_i>1$ for $i=1,\cdots,N$.
Lemma \ref{L som} in the Appendix yields $\left(\prod_{i=1}^N\frac{\partial}{\partial x_i}u\right)^{\frac{1}{N}}\in L^{\bar p} (\Om)$ taking
\begin{equation}\label{gamma esp}
 \gamma<\sum_{j=1}^N\frac{1}{\theta_j}\frac{N-\overline{p}}{\overline{p} N}
  \end{equation}
 and $\frac{\partial}{\partial x_i}u\in L^{p_i} (\Om)$  taking
\begin{equation}\label{gamma esp bis}
\gamma<\sum_{j=1}^N\frac{1}{\theta_jp_i}-\frac{1}{\theta_i}
 \end{equation}
 for $i=1,\cdots,N$. We note that \eqref{gamma esp bis} implies \eqref{gamma esp} and under latest assumption we get $u\in L^{\bar p^*}(\Om)$, but $u\not\in L^{\bar p^*+\varepsilon}(\Om)$ for every $\varepsilon>0$ (using again Lemma \ref{L som}).
\end{example}

In order to prove Proposition \ref{Prop dis} as first step we need the following result.

\begin{lemma}
For every $g_i\in C^\infty_0(\mathbb{R}^N)$, it follows that
\begin{equation}\label{Leb ineq bis}
\left\|\prod_{i=1}^Ng_i^{1/N}\right\|_{L^{1^*}}\leq \prod_{i=1}^N\left\|\frac{\partial}{\partial x_i} g_i\right\|^{1/N}_{L^{1}}.
\end{equation}
\end{lemma}

\begin{proof}
 The proof is a generalization to product of functions of the Troisi's one ( see \cite{Tr} Theorem 1.2). We have, for  $q>0$ that will be chosen later, the following inequality
\begin{equation*}
\int_{\mathbb{R}^N} \prod_{i=1}^N|g_i|^{q/N}\, dx\leq \int_{\mathbb{R}^{N}} \prod_{i=1}^N\sup_{x_i}|g_i|^{q/N}\, dx.
\end{equation*}
Since
\begin{equation*}
\int_{\mathbb{R}^{N}} \prod_{i=1}^N\sup_{x_i}|g_i|^{q/N}\, dx
\leq
\left(
\prod_{i=1}^N\int_{\mathbb{R}^{N-1}}\left[\sup_{x_i}|g_i|^{q/N}\right]
^{N-1}\, dx'
\right)
^{\frac{1}{N-1}},
\end{equation*}
where $x=(x_i,x')$ (see Lemma 4.1 of \cite{G}), we get
\begin{equation*}
\begin{split}
\!\left(\!\!\int_{\mathbb{R}^N} \prod_{i=1}^N|g_i|^{q/N}\, dx\!\right)^{\!N-1\!}\!\!\!\!&\!\!\leq \prod_{i=1}^N\int_{\mathbb{R}^{N-1}} \sup_{x_i}|g_i|^{\frac{q(N-1)}{N}}\, dx'
\leq \prod_{i=1}^N\frac{q(N-1)}{N}\int_{\mathbb{R}^{N}}|g_i|^{\frac{q(N-1)}{N}-1}\left|\frac{\partial g_i}{\partial x_i}\right|\, dx.
\end{split}
\end{equation*}
Now choosing $q$ such that $\frac{q(N-1)}{N}-1=0$, we obtain inequality \eqref{Leb ineq bis}.

\end{proof}

\medskip

\begin{proof}[Proof of Proposition \ref{Prop dis}]
Applying inequality \eqref{Leb ineq bis} to the function $$g_i=\omega_n(\phi^{\alpha_i}u^{\beta_i})+a_{n-1}^{\phi^{\alpha_i}u^{\beta_i}},$$
we obtain
\begin{equation}\label{omega_n ineq}
\left\|\left(\prod_{i=1}^N \omega_n(\phi^{\alpha_i}u^{\beta_i})+a_{n-1}^{\phi^{\alpha_i}u^{\beta_i}}\right)^{1/N}\right\|_{L^{1^*}}\leq  \prod_{i=1}^N \left\|\frac{\partial}{\partial x_i}\left[\omega_n(\phi^{\alpha_i}u^{\beta_i})+a_{n-1}^{\phi^{\alpha_i}u^{\beta_i}}\right]\right\|_{L^{1}}^{1/N},
\end{equation}
where $\omega_n$ and $a_n$ are defined as in \eqref{omega n} and \eqref{def_an}, respectively. We stress that
\begin{equation}\label{prod an}
a_n^{\phi^{\underline{\alpha}}u^{\underline{\beta}}}=\left(\prod_{i=1}^N a_n^{\phi^{\alpha_i}u^{\beta_i}}\right)^{1/N}.
\end{equation}

\noindent Moreover using the Hardy-Lyttlewood inequality, we have
\begin{equation*}
\begin{split}
&\left\|\frac{\partial}{\partial x_i}\left[\omega_n(\phi^{\alpha_i}u^{\beta_i})+a_{n-1}^{\phi^{\alpha_i}u^{\beta_i}}\right]\right\|_{L^{1}}
=
\left\|\frac{\partial}{\partial x_i}\left[\omega_n(\phi^{\alpha_i}u^{\beta_i})\right]\right\|_{L^{1}}\\
&\qquad \qquad  =
\int_{a^{\phi^{\alpha_i}u^{\beta_i}}_{n-1}< |\phi^{\alpha_i}u^{\beta_i}|\leq a^{\phi^{\alpha_i}u^{\beta_i}}_{n}}\left|\frac{\partial}{\partial x_i}(\phi^{\alpha_i}u^{\beta_i})\right|\,dx \leq \int_{0}^{k^{-(n-1)}}\left|\frac{\partial}{\partial x_i} (\phi^{\alpha_i}u^{\beta_i})\right|^*(s)\,ds.
\end{split}
\end{equation*}
Let us consider the left-hand side of \eqref{omega_n ineq}.  By \eqref{dis carat} and \eqref{prod an} we get
\begin{equation*}
\begin{split}
\left\|\left(\prod_{i=1}^N \omega_n(\phi^{\alpha_i}u^{\beta_i})+a_{n-1}^{\phi^{\alpha_i}u^{\beta_i}}\right)^{1/N}\right\|_{L^{1^*}}
&\geq
\left\|\prod_{i=1}^N
\left(a^{\phi^{\alpha_i}u^{\beta_i}}_{n}\right)^{1/N}\chi_{\{\phi^{\alpha_i}u^{\beta_i}
>a^{\phi^{\alpha_i}u^{\beta_i}}_{n}\}}
\right\|_{L^{1^*}}\\&\geq
\prod_{i=1}^N\left(a^{\phi^{\alpha_i}u^{\beta_i}}_{n}\right)^{1/N}k^{-\frac{n}{1^*}}=
a^{\phi^{\underline{\alpha}}u^{\underline{\beta}}}_{n}k^{-\frac{n}{1^*}}.\\
\end{split}
\end{equation*}

\noindent Putting all together, inequality \eqref{omega_n ineq} becomes
$$
a^{\phi^{\underline{\alpha}}u^{\underline{\beta}}}_{n}k^{-\frac{n}{1^*}}
\leq C
\prod_{i=1}^N\left(
\int_{0}^{k^{-(n-1)}}\left|\frac{\partial}{\partial x_i} (\phi^{\alpha_i}u^{\beta_i})\right|^*(s)\,ds
\right)^{1/N},$$
where $C$ denotes a positive constant that could change from line to line.

Multiplying for $k^{\frac{n}{\overline{p}'}}$, elevating to the power $\overline{p}$ and summing, we obtain
\begin{equation}\label{In1}
\begin{split}
\sum_{n \in\mathbb{Z}}\left[a^{\phi^{\underline{\alpha}}u^{\underline{\beta}}}_{n}\right]^{\overline{p}}
k^{-\frac{n\overline{p}}{\overline{p}^*}}
&
\leq
 C \sum_{n \in\mathbb{Z}}\prod_{i=1}^N\left(k^{\frac{n}{p_i'}}\int_{0}^{k^{-(n-1)}}\left|\frac{\partial}{\partial x_i} (\phi^{\alpha_i}u^{\beta_i})\right|^*(s)\,ds\right)^{\overline{p}/N}.
\end{split}
\end{equation}
On the other hand, by Proposition \ref{Th_equiv} with $p=\overline{p}^*$ and $q=\overline{p}$,  it follows that
$$\left\|\phi^{\underline{\alpha}}u^{\underline{\beta}}\right\|_{L^{\overline{p}^*,\overline{p}}}
\leq C
\left(\sum_{n \in\mathbb{Z}}\left[a_n^{\phi^{\underline{\alpha}}u^{\underline{\beta}}}\right]^{\overline{p}}
k^{-\frac{n\overline{p}}{\overline{p}^*}}
\right)^{1/\overline{p}}.
$$
We have, by the previous inequality and \eqref{In1}, that
\begin{equation*}\label{IN_2}
\left\|\phi^{\underline{\alpha}}u^{\underline{\beta}}\right\|_{L^{\overline{p}^*,\overline{p}}}
\leq C \left(\sum_{n \in\mathbb{Z}}\prod_{i=1}^N\left(k^{\frac{n}{p_i'}}
\int_{0}^{k^{-(n-1)}}\left|\frac{\partial}{\partial x_i} (\phi^{\alpha_i}u^{\beta_i})\right|^*(s)\,ds\right)^{\overline{p}/N}
\right)^{1/\overline{p}}.
\end{equation*}
We observe that by  \eqref{u star star}, since $k>1$  it follows that
\begin{equation*}
\begin{split}
\left\|\phi^{\underline{\alpha}}u^{\underline{\beta}}\right\|_{L^{\overline{p}^*,\overline{p}}}\leq &C \left(\sum_{n \in\mathbb{Z}}\prod_{i=1}^N\left(k^{\frac{n}{p_i'}}
\int_{0}^{k^{-(n-1)}}\left|\frac{\partial}{\partial x_i} (\phi^{\alpha_i}u^{\beta_i})\right|^*(s)\,ds\right)^{\overline{p}/N}
\right)^{1/\overline{p}}\\
=&\left(\sum_{n \in\mathbb{Z}}\prod_{i=1}^N\left(k^{1-\frac{n}{p_i}}
\left|\frac{\partial}{\partial x_i} (\phi^{\alpha_i}u^{\beta_i})\right|^{**}({k^{-(n-1)}})\right)^{\overline{p}/N}\right)^{1/\overline{p}}\\
=&\left(\sum_{n \in\mathbb{Z}}\frac{k^{\overline{p}}}{\log k}\prod_{i=1}^N\left(
\left|\frac{\partial}{\partial x_i} (\phi^{\alpha_i}u^{\beta_i})\right|^{**}({k^{-(n-1)}})\right)^{\overline{p}/N} \int_{k^{-n}}^{k^{-(n-1)}}\frac{ds}{s}\right)^{1/\overline{p}}\\
\leq &\left(\frac{k^{\overline{p}}}{\log k}\right)^{1/\overline{p}}\left(\sum_{n \in\mathbb{Z}}\int_{k^{-n}}^{k^{-(n-1)}}\prod_{i=1}^{N}\left(\left|\frac{\partial}{\partial x_i} (\phi^{\alpha_i}u^{\beta_i})\right|^{**}(s)\right)^{\overline{p}/N}\frac{ds}{s}\right)^{1/\overline{p}}\\
\leq &C
\left\|\left(\prod_{i=1}^{N}\frac{\partial}{\partial x_i} (\phi^{\alpha_i}u^{\beta_i})\right)^{1/N}\right\|_{L^{\overline{p}}},
\end{split}
\end{equation*}
that is \eqref{Lor ineq bis zero}.

\end{proof}

We stress that using inequality \eqref{Leb ineq bis} we can prove the following well-known anisotropic Sobolev inequalities in Lebesgue spaces. For every $u\in C^\infty_0(\mathbb{R}^N)$ there exists two positive constant $C_1,C_2$ independent of $u$ such that
$$\left\|u\right\|_{L^{\bar p^*}}\leq C_1\prod_{i=1}^N\left\|\frac{\partial}{\partial x_i} u\right\|^{1/N}_{L^{p_i}} \quad \text { when } \bar p<N$$
$$\left\|u\right\|_{L^{\bar q}}\leq  C_2\left[\left\|u\right\|_{L^{p_0}}
+\prod_{i=1}^N\left\|\frac{\partial}{\partial x_i} u\right\|^{1/N}_{L^{p_i}}\right]\quad \text { when } \bar p=N$$
for every $q\in [p_0,\infty)$ and $p_0\geq 1$. Indeed the starting point is to take $g_i=u^{\sigma_i}$ in \eqref{Leb ineq bis} with $\sum_{i=1}^N \sigma_i=Nt$ for suitable $t>0$.

\section{Proof of Theorem \ref{mainTHM}}\label{Sez proof 1}

 The proof consists in two steps. First we prove our regularity result assuming that $\|b_i\|_{L^{\frac{Np'_i}{\overline{p}},\infty}}$ are small enough for $i=1,\cdots,N$. Then,  last assumption is removed thanks \eqref{hpdistmain}.

\medskip

\textit{Step 1. Proof assuming that $\|b_i\|_{L^{\frac{Np'_i}{\overline{p}},\infty}}$ are small enough. }
Let us fix $\varphi\in C_0^1(\mathbb{R}^N), \varphi\geq0$, supported in a ball contained in $\Omega$.
Let $v=u-T_k(u)$, where $T_k$ is defined as in \eqref{T}.
for all $s \in \mathbb{R}$. We use $w=\varphi^qv$ with $q>0$ (that we will choose later) as test function in the weak formulation \eqref{sol1} and we denote by $\Omega_k=\{|u|>k\}$.
Using assumptions \eqref{ii1}, \eqref{ii2} and \eqref{b1}, we get
\begin{align}\label{I0}
\alpha \sum_{i=1}^N\int_{\Omega_k}\!\left|\frac{\partial}{\partial x_i}v\right|^{p_i}\!\!\varphi^q dx
&
\leq \sum_{i=1}^N \beta_i\!\! \int_{\Omega_k} \left|\frac{\partial}{\partial x_i}v\right|^{p_i-1}\!\!\varphi^{q-1}
\left|\frac{\partial}{\partial x_i}\varphi\right|v dx+\!\sum_{i=1}^N\int_{\Omega_k}\!\! |b_i| |u|^{\frac{\overline{p}}{p_i'}}\left|\frac{\partial}{\partial x_i}(v\varphi^q)\right| dx
\nonumber\\
&+\sum_{i=1}^N\int_{\Omega_k} |\mathcal F_i|^{p_i-1}\left|\frac{\partial}{\partial x_i}v\right|\varphi^q  dx
+\sum_{i=1}^N\int_{\Omega_k} |\mathcal F_i|^{p_i-1}v\left|\frac{\partial}{\partial x_i}\varphi\right|q\varphi^{q-1} dx \nonumber\\
&\leq \sum_{i=1}^N  \!\beta_i \!\int_{\Omega_k}\! \left|\frac{\partial}{\partial x_i}u\right|^{p_i-1}\!\!\varphi^{q-1}
\left|\frac{\partial}{\partial x_i}\varphi\right|v dx\!+C\sum_{i=1}^N\int_{\Omega_k}\!\! |b_i| k^{\frac{\overline{p}}{p_i'}}\!\left|\frac{\partial}{\partial x_i}(v\varphi^q)\right| dx \nonumber\\
&+C\sum_{i=1}^N\int_{\Omega_k} |b_i| |v|^{\frac{\overline{p}}{p_i'}}\left|\frac{\partial}{\partial x_i}v\right|\varphi^q dx
+C\sum_{i=1}^N\int_{\Omega_k} |b_i| |v|^{\frac{\overline{p}}{p_i'}}v\left|\frac{\partial}{\partial x_i}\varphi\right|q\varphi^{q-1} dx \nonumber
\\
&+\sum_{i=1}^N\int_{\Omega_k} |\mathcal F_i|^{p_i-1}\left|\frac{\partial}{\partial x_i}v\right|\varphi^q dx
+\sum_{i=1}^N\int_{\Omega_k} |\mathcal F_i|^{p_i-1}v\left|\frac{\partial}{\partial x_i}\varphi\right|q\varphi^{q-1}dx \nonumber\\
&:=I_1+I_2+I_3+I_4+I_5+I_6,
\end{align}
where $C$ is a positive constant independent of $u$. By Young inequality, we have
\begin{equation}\label {I1}
I_1
\leq
\varepsilon \sum_{i=1}^N\int_{\Omega_k} \left|\frac{\partial}{\partial x_i}v\right|^{p_i}\varphi^q
+C(\varepsilon)
\sum_{i=1}^N\int_{\Omega_k} \varphi^{q-p_i}
\left|\frac{\partial}{\partial x_i}\varphi\right|^{p_i}v^{p_i},
\end{equation}
\noindent where here and in what follows $\varepsilon$ denotes a positive constant that will be chosen later. Using Young inequality, H\"{o}lder inequality and \eqref{Lor ineq bis}, we get
\begin{equation}\label {I2}
\begin{split}
  I_3
\leq
&
\varepsilon \sum_{i=1}^N\int_{\Omega_k} \left|\frac{\partial}{\partial x_i}v\right|^{p_i}\varphi^q dx
+C(\varepsilon)
\sum_{i=1}^N\int_{\Omega_k} |b_i|^{p'_i}|v|^{\overline{p}}\varphi^{q} dx\\
\leq &
\varepsilon \sum_{i=1}^N\int_{\Omega_k} \left|\frac{\partial}{\partial x_i}v\right|^{p_i}\varphi^q dx
+C(\varepsilon)
\sum_{i=1}^N \|b_i\|^{p'_i}_{L^{\frac{Np'_i}{\overline{p}},\infty}} \|v \varphi^{\frac{q}{\overline{p}}}\|^{\overline{p}}_{L^{\overline{p}^{*},\overline{p}}}\\
\leq &
\varepsilon \sum_{i=1}^N\int_{\Omega_k} \left|\frac{\partial}{\partial x_i}v\right|^{p_i}\varphi^q dx
+C(\varepsilon)
\sum_{i=1}^N \|b_i\|^{p'_i}_{L^{\frac{Np'_i}{\overline{p}},\infty}} \left(\prod_{i=1}^N\left\| \frac{\partial}{\partial x_i}\left(v\varphi^{\frac{q}{p_i}}\right)\right\|_{L^{p_i}}^{\frac{1}{N}} \right)^{\overline{p}}
\end{split}
\end{equation}

\noindent and
\begin{equation}\label{I4}
  I_4
\leq
\varepsilon \sum_{i=1}^N\!\int_{\Omega_k} v^{p_i}\left|\frac{\partial}{\partial x_i}\!\!\varphi\right|^{p_i}\varphi^{q-p_i}\! dx
\!+\!C(\varepsilon)\sum_{i=1}^N \|b_i\|^{p'_i}_{L^{\frac{Np'_i}{\overline{p}},\infty}}\!\! \left(\prod_{i=1}^N\left\| \frac{\partial}{\partial x_i}\left(v\varphi^{\frac{q}{p_i}}\right)\right\|_{L^{p_i}}^{\frac{1}{N}} \right)^{\overline{p}}.
\end{equation}
Finally Young and H\"{o}lder inequality yields
\begin{equation}\label{I6}
\!\!\!\! \!I_5
 +I_6
\leq
\varepsilon \sum_{i=1}^N\int_{\Omega_k} \!\!\!v^{p_i}\left|\frac{\partial}{\partial x_i}\varphi\right|^{p_i}\!\!\varphi^{q-p_i} dx+\varepsilon \sum_{i=1}^N\!\int_{\Omega_k} \!\left|\frac{\partial}{\partial x_i}v\right|^{p_i} \!\!\varphi^{q} dx
+C(\varepsilon) \sum_{i=1}^N\!\int_{\Omega_k} |\mathcal F_i|^{p_i}\!\varphi^q dx.
\end{equation}

\noindent Now putting together inequalities \eqref{I1}, \eqref{I2}, \eqref{I4} and \eqref{I6}, rearranging and choosing $\varepsilon$ small enough then inequality \eqref{I0} becomes
\begin{align}\label{I_finale}
\!\! \sum_{i=1}^N\!\int_{\Omega_k} \!\left|\frac{\partial}{\partial x_i}v\right|^{p_i}\!\! \varphi^{q}\!dx
\!\leq &
C\!\left[\sum_{i=1}^N\!\int_{\Omega_k} \!\!|b_i| k^{\frac{\overline{p}}{p_i'}}\left|\frac{\partial}{\partial x_i}(v\varphi^q)\right| dx\right.
\!+\!\sum_{i=1}^N \|b_i\|^{p'_i}_{L^{\frac{Np'_i}{\overline{p}},\infty}}\!\!\! \left(\prod_{i=1}^N\left\| \frac{\partial}{\partial x_i}\left(v\varphi^{\frac{q}{p_i}}\right)\right\|_{L^{p_i}}^{\frac{1}{N}} \!\right)^{\overline{p}}
\nonumber \\
&
+ \left.\sum_{i=1}^N\int_{\Omega_k} v^{p_i}\left|\frac{\partial}{\partial x_i}\varphi\right|^{p_i}\varphi^{q-p_i} dx
+ \sum_{i=1}^N\int_{\Omega_k} |\mathcal F_i|^{p_i}\varphi^q dx\right],
\end{align}
for suitable constant $C=C(\vec{p},\alpha, \vec{\beta}, N)$ that from now on could change from line to line.
Now we note that
\begin{equation*}
\begin{split}
\sum_{i=1}^N&\int_{\Omega_k} \left|\frac{\partial}{\partial x_i}\left(v\varphi^{\frac{q}{p_i}}\right)\right|^{p_i} dx
\leq
\sum_{i=1}^N\int_{\Omega_k} 2^{p_i-1}\left|\frac{\partial}{\partial x_i}v\right|^{p_i} \varphi^{q}dx+ \sum_{i=1}^N\int_{\Omega_k} 2^{p_i-1} v^{p_i}\left|\frac{\partial}{\partial x_i}\varphi\right|^{p_i}\varphi^{q-p_i} dx\\
 &\leq
C\left[\sum_{i=1}^N\int_{\Omega_k} |b_i| k^{\frac{\overline{p}}{p_i'}}\left|\frac{\partial}{\partial x_i}(v\varphi^q)\right| dx+\sum_{i=1}^N \|b_i\|^{p'_i}_{L^{\frac{Np'_i}{\overline{p}},\infty}} \left(\prod_{i=1}^N\left\| \frac{\partial}{\partial x_i}\left(v\varphi^{\frac{q}{p_i}}\right)\right\|_{L^{p_i}}^{\frac{1}{N}} \right)^{\overline{p}}\right.\\
&
\left.+ \sum_{i=1}^N\int_{\Omega_k}  v^{p_i}\left|\frac{\partial}{\partial x_i}\varphi\right|^{p_i}\varphi^{q-p_i} dx + \sum_{i=1}^N\int_{\Omega_k} |\mathcal F_i|^{p_i}\varphi^q dx\right],
\end{split}
\end{equation*}
and so
\begin{align}\label{stima A}
\left(\int_{\Omega_k} \left|\frac{\partial}{\partial x_j}\left(v\varphi^{\frac{q}{p_j}}\right)\right|^{p_j} dx\right)^{\frac{1}{p_j}}
\leq &
\left( \sum_{i=1}^N\int_{\Omega_k} \left|\frac{\partial}{\partial x_i}\left(v\varphi^{\frac{q}{p_i}}\right)\right|^{p_i} dx\right)^{\frac{1}{p_j}}
\\
 &\leq C^{\frac{1}{p_j}}
\left[\sum_{i=1}^N\int_{\Omega_k} |b_i| k^{\frac{\overline{p}}{p_i'}}\left|\frac{\partial}{\partial x_i}(v\varphi^q)\right| dx
+\sum_{i=1}^N \|b_i\|^{p'_i}_{L^{\frac{Np'_i}{\overline{p}},\infty}} A^{\frac{\overline{p}}{N}}\right.
\nonumber\\
&
+ \left.\sum_{i=1}^N\int_{\Omega_k} v^{p_i}\left|\frac{\partial}{\partial x_i}\varphi\right|^{p_i}\varphi^{q-p_i} dx + \sum_{i=1}^N\int_{\Omega_k} |\mathcal F_i|^{p_i}\varphi^q dx\right]^{\frac{1}{p_j}},
\nonumber
\end{align}
where $A=\prod_{i=1}^N\left\| \frac{\partial}{\partial x_i}\left(v\varphi^{\frac{q}{p_i}}\right)\right\|_{L^{p_i}}$.   Making the product on the left and right sides of \eqref{stima A}, we get
\begin{equation}\label{stima prodotti_A}
\begin{split}
A
\leq &
\,\,C\left[\left(\sum_{i=1}^N\int_{\Omega_k} |b_i| k^{\frac{\overline{p}}{p_i'}}\left|\frac{\partial}{\partial x_i}(v\varphi^q)\right| dx\right)^{\frac{N}{\overline{p}}}+\left(\sum_{i=1}^N \|b_i\|^{p'_i}_{L^{\frac{Np'_i}{\overline{p}},\infty}}\right)^{\frac{N}{\overline{p}}} A\right.\\
&
+ \left.\left(\sum_{i=1}^N\int_{\Omega_k} v^{p_i}\left|\frac{\partial}{\partial x_i}\varphi\right|^{p_i}\varphi^{q-p_i} dx\right)^{\frac{N}{\overline{p}}} +\left( \sum_{i=1}^N\int_{\Omega_k} |\mathcal F_i|^{p_i}\varphi^q dx\right)^{\frac{N}{\overline{p}}}\right].
\end{split}
\end{equation}
The assumption that the norms $\|b_i\|^{p'_i}_{L^{\frac{Np'_i}{\overline{p}},\infty}}$ are small than a suitable constant depending on $\vec{p},\alpha, \vec{\beta}, N$
 and \eqref{stima prodotti_A} allow us to obtain
\begin{equation}\label{stima prod_A_fin}
  \begin{split}
  A \leq &
  C\left[ \left(\sum_{i=1}^N\int_{\Omega_k} |b_i| k^{\frac{\overline{p}}{p_i'}}\left|\frac{\partial}{\partial x_i}(v\varphi^q)\right| dx\right)^{\frac{N}{\overline{p}}}+\left(\sum_{i=1}^N\int_{\Omega_k} v^{p_i}\left|\frac{\partial}{\partial x_i}\varphi\right|^{p_i}\varphi^{q-p_i} dx\right)^{\frac{N}{\overline{p}}}\right.\\
  &
\left.+\left( \sum_{i=1}^N\int_{\Omega_k} |\mathcal F_i|^{p_i}\varphi^q dx\right)^{\frac{N}{\overline{p}}}\right].
  \end{split}
\end{equation}
Substituting \eqref{stima prod_A_fin} in \eqref{I_finale}, by easy calculations it follows
\begin{equation*}\label{I_finaleBIS}
\begin{split}
\sum_{i=1}^N\int_{\Omega_k} \left|\frac{\partial}{\partial x_i}v\right|^{p_i} \varphi^{q}dx
\leq &
C\left(\sum_{i=1}^N\int_{\Omega_k} b_i k^{\frac{\overline{p}}{p_i'}}\left|\frac{\partial}{\partial x_i}(v\varphi^q)\right| dx
+\sum_{i=1}^N\int_{\Omega_k}  |v|^{p_i}\left|\frac{\partial}{\partial x_i}\varphi\right|^{p_i}\varphi^{q-p_i} dx\right.\\
&+\left.\sum_{i=1}^N\int_{\Omega_k} |\mathcal F_i|^{p_i}\varphi^q dx\right).
\end{split}
\end{equation*}

At this point we multiply both sides of previous inequality by $k^{\gamma}$ for fixed $\gamma>0 $ that will be chosen later and we integrate with respect to $k$ over $[0, K]$ for fixed $K>0$.  A repeated use of Fubini's theorem gives

\begin{align}\label{I_finaleBIS2}
\sum_{i=1}^N\int_{\Omega} &\left|\frac{\partial}{\partial x_i}u\right|^{p_i}\varphi^{q} |T_Ku|^{\gamma+1}dx\leq  C\left[\sum_{i=1}^N\int_{\Omega} b_i |T_Ku|^{\frac{\overline{p}}{p_i'} +\gamma+1} \left(  \left|\frac{\partial}{\partial x_i}u\right|  \varphi^q  +  q \varphi^{q-1}
\left|u\frac {\partial}{\partial x_i} \varphi \right| \right)dx
\right.
\nonumber\\
&\left. +\sum_{i=1}^N\int_{\Omega}  |u|^{p_i}\left|\frac{\partial}{\partial x_i}\varphi\right|^{p_i}\varphi^{q-p_i} |T_Ku|^{\gamma+1}  dx
+\sum_{i=1}^N\int_{\Omega} |\mathcal F_i|^{p_i}\varphi^q|T_Ku|^{\gamma+1} dx\right]
\end{align}
where $C=C(\gamma,\alpha,N,\vec{p},\vec{\beta})$. Now by H\"{o}lder inequality and Young inequality we get
  \begin{equation*}
  \begin{split}
    \sum_{i=1}^N\int_{\Omega} b_i |T_Ku|^{\frac{\overline{p}}{p_i'} +\gamma+1} &\left(  \left|\frac{\partial}{\partial x_i}u\right|  \varphi^q  +  q \varphi^{q-1} \left|u\frac {\partial}{\partial x_i} \varphi \right| \right)dx \leq\\
    &
     \left(\sum_{i=1}^N\int_{\Omega}\left|\frac{\partial}{\partial x_i}u\right|^{p_i} |T_Ku|^{\gamma+1} \varphi^q dx\right)^{\frac{1}{p_i}}
     \left( \sum_{i=1}^N\int_{\Omega} b_i^{p'_i} |T_Ku|^{\gamma+1+\overline{p}} \varphi^q dx\right)^{\frac{1}{p'_i}}\\
     & +\left(\sum_{i=1}^N\int_{\Omega}\left|u\frac{\partial}{\partial x_i}\varphi\right|^{p_i} |T_Ku|^{\gamma+1} \varphi^{q-p_i} dx\right)^{\frac{1}{p_i}} \left( \sum_{i=1}^N\int_{\Omega} b_i^{p'_i} |T_Ku|^{\gamma+1+\overline{p}} \varphi^q dx\right)^{\frac{1}{p'_i}}\\
     & \leq \epsilon \sum_{i=1}^N\int_{\Omega}\left|\frac{\partial}{\partial x_i}u\right|^{p_i} |T_Ku|^{\gamma+1} \varphi^q dx+
     \epsilon \sum_{i=1}^N\int_{\Omega}\left|u\frac{\partial}{\partial x_i}\varphi\right|^{p_i} |T_Ku|^{\gamma+1} \varphi^{q-p_i} dx\\
     & + 2C(\epsilon) \|b_i\|_{L^{\frac{Np'_i}{\overline{p}},\infty}} \|T_Ku^{\frac{\gamma+1}{\overline{p}}+1}\varphi^{\frac{q}{\overline{p}}}\|^{\overline{p}}_{L^{\overline{p}^*,\overline{p}}},
\end{split}
  \end{equation*}
 where $\epsilon$ is a positive constant small enough.  Substituting the previous inequality in \eqref{I_finaleBIS2} and by \eqref{Lor ineq bis}, we obtain
  \begin{align}\label{l_finaleBIS3}
\sum_{i=1}^N\int_{\Omega}& \left|\frac{\partial}{\partial x_i}u\right|^{p_i} \varphi^{q} |T_Ku|^{\gamma+1}dx\leq
 C\left[ \sum_{i=1}^N\int_{\Omega}\left|u\frac{\partial}{\partial x_i}\varphi\right|^{p_i} |T_Ku|^{\gamma+1} \varphi^{q-p_i} dx\right.\\
 &+ \sum_{i=1}^N\|b_i\|_{L^{\frac{Np'_i}{\overline{p}},\infty}} \left(\prod_{i=1}^N\left\| \frac{\partial}{\partial x_i}\left(|T_Ku|^{\frac{\gamma+1}{p_i}+1}\varphi^{\frac{q}{p_i}}\right)\right\|_{L^{p_i}}^{\frac{1}{N}} \right)^{\overline{p}}
 \left. +  \sum_{i=1}^N\int_{\Omega} |\mathcal F_i|^{p_i}\varphi^q|T_Ku|^{\gamma+1} dx\right].\nonumber
  \end{align}
  Denoting $B= \prod_{i=1}^N\left\| \frac{\partial}{\partial x_i}\left(|T_Ku|^{\frac{\gamma+1}{p_i}+1}\varphi^{\frac{q}{p_i}}\right)\right\|_{L^{p_i}}$, it follows
  \begin{align}\label{stima B}
&   \sum_{i=1}^N \int_{\Omega} \left|\frac{\partial}{\partial x_i} \left(|T_Ku|^{\frac{\gamma+1}{p_i}+1}\varphi^{\frac{q}{p_i}}\right)\right|^{p_i} dx\\
  & \leq
   \sum_{i=1}^N\int_{\Omega} 2^{p_i-1}\left|\frac{\partial}{\partial x_i}u\right|^{p_i} \varphi^{q} |T_Ku|^{\gamma+1}dx
  + \sum_{i=1}^N\int_{\Omega}2^{p_i-1}\left|u\frac{\partial}{\partial x_i}\varphi\right|^{p_i} |T_Ku|^{\gamma+1} \varphi^{q-p_i} dx
  \nonumber\\
  & \leq C \left[ \sum_{i=1}^N\int_{\Omega}\left|u\frac{\partial}{\partial x_i}\varphi\right|^{p_i} |T_Ku|^{\gamma+1} \varphi^{q-p_i} dx
 + \sum_{i=1}^N   \|b_i\|_{L^{\frac{Np'_i}{\overline{p}},\infty}} B^{\frac{\overline{p}}{N}}+  \sum_{i=1}^N\int_{\Omega} |\mathcal F_i|^{p_i}\varphi^q|T_Ku|^{\gamma+1} dx\right].\nonumber
  \end{align}
%
  Previous inequality  give us an estimate of the $j-th$ addendum of the sum at the left-hand side of \eqref{stima B} as well. Then elevating to the power $\frac 1{p_j}$, making the product on the left and right sides of \eqref{stima B}, we get
\begin{equation}\label{stima prod B}
 \begin{split}
B
 &\leq C \left[ \left( \sum_{i=1}^N\int_{\Omega} \left|u\frac{\partial}{\partial x_i}\varphi\right|^{p_i} |T_Ku|^{\gamma+1} \varphi^{q-p_i} dx\right)^{\frac{N}{\overline{p}}}+ \left( \sum_{i=1}^N \|b_i\|_{L^{\frac{Np'_i}{\overline{p}},\infty}}\right)^{\frac{N}{\overline{p}}}  B\right.\\
 &+ \left.\left( \sum_{i=1}^N\int_{\Omega} |\mathcal F_i|^{p_i}\varphi^q|T_Ku|^{\gamma+1} dx\right)^{\frac{N}{\overline{p}}}\right].
   \end{split}
  \end{equation}
Using again that the norms $\|b_i\|^{p'_i}_{L^{\frac{Np'_i}{\overline{p}},\infty}}$ are small than a suitable constant depending on $\gamma,\alpha,N,\vec{p}, \vec{\beta}$
 from \eqref{stima prod B} we obtain

  \begin{equation}\label{stima prod B_fin}
 \!\!\! B\leq \!C\!\left[ \left( \sum_{i=1}^N\int_{\Omega} \left|u\frac{\partial}{\partial x_i}\varphi\right|^{p_i} \!|T_Ku|^{\gamma+1} \!\!\varphi^{q-p_i} dx\right)^{\frac{N}{\overline{p}}}\!\!+\!\left( \sum_{i=1}^N\int_{\Omega} |\mathcal F_i|^{p_i}\varphi^q|T_Ku|^{\gamma+1} dx\right)^{\frac{N}{\overline{p}}}\!\right].\!
  \end{equation}
  Substituting \eqref{stima prod B_fin} in \eqref{l_finaleBIS3} it follows that

    \begin{equation}\label{punto finale}
  \begin{split}
 &\sum_{i=1}^N\left\| \frac{\partial}{\partial x_i}\left(|T_Ku|^{\frac{\gamma+1}{p_i}+1}\varphi^{\frac{q}{p_i}}\right)\right\|_{L^{p_i}}^{p_i}\\
 & \leq C\left[ \sum_{i=1}^N\int_{\Omega}\left|u\frac{\partial}{\partial x_i}\varphi\right|^{p_i} |T_Ku|^{\gamma+1} \varphi^{q-p_i} dx+  \sum_{i=1}^N\int_{\Omega} |\mathcal F_i|^{p_i}\varphi^q|T_Ku|^{\gamma+1} dx\right].
  \end{split}
  \end{equation}

At this point, let us assume that  {{$u\in L^{\mu p_{\max}}_{loc}(\Omega)$ }, where $\mu$ is defined in \eqref{mu}. Recalling that  $\mathcal F_i\in L^{r_i}_{loc}(\Omega)$ for every $i=1,...,N$, } then $\mathcal F_i \in L_{loc}^{p_i\mu}(\Om)$. Hence by H\"{o}lder inequality with exponents $\mu$ and $\mu'$, from \eqref{punto finale} we have
  \begin{equation}\label{stima regolarita}
  \begin{split}
  \prod_{i=1}^N &\left\| \frac{\partial}{\partial x_i} \left(|T_Ku|^{\frac{\gamma+1}{p_i}+1}  \varphi^{\frac{q}{p_i}}\right)\right\|_{L^{p_i}}^{\frac{1}{N}} \leq     C  \left( \int_{\Omega} |T_Ku|^{(\gamma+1) \frac{\mu}{\mu-1}} \varphi^{\frac{\overline{p}^* q}{\overline{p}}} dx\right)^{\frac{1}{\overline{p}}    \left(1-\frac 1\mu \right)}\\
 &\qquad\qquad\qquad \times\left[  \sum_{i=1}^N \left( \int_{\Omega}\left|u\frac{\partial}{\partial x_i}\varphi\right|^{p_i  \mu }  \varphi^{\alpha_1} dx\right)^{\frac 1{\bar p \mu}} + \sum_{i=1}^N  \left( \int_{\Omega} |\mathcal F_i|^{p_i \mu }\varphi^{\alpha_2} dx\right)^{\frac{1}{\overline{p} \mu }} \right]\\
\end{split}
  \end{equation}
with $\alpha_1= \left[  1- \frac{\overline{p}^*}{\overline{p}} \left(  1-\frac 1\mu \right ) \right]   \mu q - p_i \mu $,
    and $\alpha_2=  \left[ 1- \frac{\overline{p}^*}{\overline{p}} \left( 1-\frac 1\mu \right)  \right] q\mu$. Now, by  \eqref{hp-ri} we have $\alpha_2>0 $  and we can now choose $q$ large enough to have $\alpha_1>0$. Finally,  we choose now $\gamma =\frac{N\bar p}{N-\bar p \mu} (\mu -1)-1$
  so that

  \begin{equation*}\label{gamma}
 (\gamma+1)   \left( \frac{\mu}{ \mu-1} \right) \frac{ \overline{p}}{\gamma+1+\overline{p}}=\overline{p}^*.
  \end{equation*}
 Thus, using Proposition \ref{Prop dis}, \eqref{stima regolarita} becomes
\begin{equation*}
\begin{split}
&\prod_{i=1}^N\left\| \frac{\partial}{\partial x_i}\left(|T_Ku|^{\frac{\gamma+1}{p_i}+1}\varphi^{\frac{q}{p_i}}\right)\right\|_{L^{p_i}}^{\frac{1}{N}}\\
& \leq  C \left\||T_Ku|^{\frac{\gamma+1}{\overline{p}}+1}
  \varphi^{\frac{q}{\overline{p}}}\right\|_{L^{\overline{p}^*}}^{ \left(1-\frac 1\mu \right) \frac{\overline{p}^*}{\overline{p}}   }
   \left( \sum_{i=1}^N\left\|\left|u\frac{\partial}{\partial x_i}\varphi\right| \varphi^{\frac{\alpha_1}{p_i\mu }}  \right\|_{L^{p_i\mu }}^{\frac{p_i}{\overline{p}}}+\sum_{i=1}^N \left\||\mathcal F_i|\varphi^{\frac{\alpha_2}{p_i\mu }}\right\|_{L^{p_i \mu  }}^{\frac{
  p_i }{\overline{p}}} \right)\\
    &\leq C\prod_{i=1}^N\left\| \frac{\partial}{\partial x_i}\left(|T_Ku|^{\frac{\gamma+1}{p_i}+1}\varphi^{\frac{q}{p_i}}\right)\right\|_{L^{p_i}}^{ \left(1-\frac 1\mu \right) \frac{\overline{p}^*}{\overline{p}N}}
  \left(
  \sum_{i=1}^N\left\|\left|u\frac{\partial}{\partial x_i}\varphi\right| \varphi^{\frac{\alpha_1}{p_i\mu }}  \right\|_{L^{p_i\mu }}^{\frac{p_i}{\overline{p}}}
 +\sum_{i=1}^N \left\||\mathcal F_i|\varphi^{\frac{\alpha_2}{p_i \mu  }}\right\|_{L^{p_i \mu  } }^{\frac{
  p_i}{\overline{p}}}
   \right).
\end{split}
\end{equation*}
 Rearranging the previous inequality it follows that
\begin{equation}\label{limite}
\begin{split}
 \left(\prod_{i=1}^N\left\| \frac{\partial}{\partial x_i}\left(|T_Ku|^{\frac{\gamma+1}{p_i}+1}\varphi^{\frac{q}{p_i}}\right)\right\|_{L^{p_i}}^{\frac{1}{N}}\right)^\rho&
 \leq C  \left(
  \sum_{i=1}^N\left\|\left|u\frac{\partial}{\partial x_i}\varphi\right| \varphi^{\frac{\alpha_1}{p_i\mu }}  \right\|_{L^{p_i\mu }}^{\frac{p_i}{\overline{p}}}
 +\sum_{i=1}^N \left\||\mathcal F_i|\varphi^{\frac{\alpha_2}{p_i \mu  }}\right\|_{L^{p_i \mu  } }^{\frac{
  p_i}{\overline{p}}}
   \right)\\
\end{split}
\end{equation}
 where $\rho:= \left[  1- \frac{\overline{p}^*}{\overline{p}} \left(  1-\frac 1\mu \right ) \right] $ that is positive by  \eqref{hp-ri}. Finally letting $K\rightarrow +\infty$ in \eqref{limite}, we conclude
\begin{equation*}
\begin{split}
 \left(\prod_{i=1}^N\left\| \frac{\partial}{\partial x_i}\left(|u|^{\frac{\gamma+1}{p_i}+1}\varphi^{\frac{q}{p_i}}\right)\right\|_{L^{p_i}}^{\frac{1}{N}}\right)^\rho&
  \leq C  \left(
  \sum_{i=1}^N\left\|\left|u\frac{\partial}{\partial x_i}\varphi\right| \varphi^{\frac{\alpha_1}{p_i\mu }}  \right\|_{L^{p_i\mu }}^{\frac{p_i}{\overline{p}}}
 +\sum_{i=1}^N \left\||\mathcal F_i|\varphi^{\frac{\alpha_2}{p_i \mu  }}\right\|_{L^{p_i \mu  } }^{\frac{
  p_i}{\overline{p}}}
   \right).\\
\end{split}
\end{equation*}
 Since $\left( \frac{\gamma+1}{\bar p}+1 \right)\bar p^*= \frac{N\bar p \mu}{ N-\bar p\mu}$, using again Proposition \ref{Prop dis}  we obtain that

 \begin{equation}\label{def s}
 u\in L^{s(\mu)}_{loc}(\Om)\qquad \text{ with } s(\mu):=(\bar p \mu)^*
 .
 \end{equation}

  We observe that previous argument works directly when  $\mu p_{\max}\leq \overline{p}^*$, since in this case $u\in W^{1,\vec p}_{loc}(\Om)$ implies that $u\in L^{\mu p_i}_{loc}(\Om)$ for every $i=1,....,N$.  If otherwise there exist $i\in \{1,...,N\}$ such that $\mu p_i> \overline{p}^*$ we use a bootstrap procedure.

 Precisely, if there exist
  $i_j \in \{1,...,N\}$ such that
$\mu p_{i_j}> \overline{p}^*$ , $j=1,...,m$, with $m\leq N$, we repeat the previous argument  (from \eqref{punto finale} to \eqref{def s} )   with $r_{i_j} $, $j=1,...,m$ replaced by $s(1)= \overline{p}^*$ and $\mu$ replaced by the corresponding new minimum  in \eqref{mu}, \emph{i.e.} $ \mu_1= \frac{\bar p^*}{p_{\max}}.$   In this way   we find  $u\in L^{s(\mu_1)}_{loc}(\Om)$ 
and
\[
s(\mu_1)-s(1) > \bar p^* - p_{\max}.
\]

At this point, if $\mu p_{\max}\leq s(\mu_1)$ we can use  the information $u\in L^{s(\mu_1)}_{loc}(\Om)$ to conclude our proof as before. Otherwise we repeat the procedure again, that is: if there exist
  $i_{j} \in \{1,...,N\}$ such that
$\mu p_{i_{j}}>s(\mu_1)$,  we repeat the previous argument  (again, from \eqref{punto finale} to \eqref{def s} )   with $r_{i_{j}} $ replaced by $s(\mu_1)$ and  $\mu$ repleaded by $\mu_2:= \frac{s(\mu_1)}{p_{\max} }$ so that   we find   $u\in L^{s(\mu_2)}_{loc}(\Om)$  with $s(\mu_2) =   \left( \bar p  \, \frac{ s(\mu_1)} {p_{\max}}\right)  ^*$.
Using the convexity of the function $f(p) =\frac{Np}{N-p}$ we find
\[
s(\mu_2)-s(\mu_1) \ge \frac{N^2}{(N-\bar p)^2}\frac{\bar p} {p_{\max}} (s(\mu_1) - \bar p^*)>\frac{\bar p^*} {p_{\max}}  \left(
\bar p^* -p_{\max}\right) ,
\]
and so on. Since, if necessary, at the $h$-th step one has  $\mu_h=\frac{s(\mu_{h-1})}{p_{\max}}$ and
\[
s(\mu_h)-s(\mu_{h-1})> \left( \frac{\bar p^*} {p_{\max}}\right)^{h-1} \left(
\bar p^* -p_{\max}\right),
\]
it is now clear that, in a finite number of times, we can conclude the proof of Step 1.

\noindent\textit{Step 2. Dropping the smallness assumption}

 Now we remove the smallness assumptions on   $\|b_i\|_{L^{\frac{Np'_i }{\overline{p}},\infty}}$ assuming  that

 \begin{equation*}\label{distanza step}
\max_i \left\{ { \rm{dist}}_{L^{\frac{N p'_i }{\bar p},\infty}(\Om)} (b_i, L^\infty(\Om))\right\}
\end{equation*}
 is sufficiently small. Setting
 \begin{equation*}\label{theta}
\theta_{i}(x)=
\left\{
           \begin{array}{ll}
            \frac{T_M b_i(x)}{b_i(x)} & \hbox{ if } b_i(x)\neq0  \\
             1  & \hbox{  if }b_i(x)=0  \\
           \end{array}
         \right.
\end{equation*}
for $M>0,$   we  rewrite equation \eqref{sol1} in the following form
  \begin{equation}\label{solthetabis}
\!\!\Si \!\int_\Om \!\!  \left (\mathcal{A}_i(x,\nabla u) \!+\! (1-\theta_i (x) )\mathcal B_i(x,u) \right) \dei \varphi \, dx\! =\! \Si \!\int_\Om  \!( \mathcal |F_i|^{p_i-2} \mathcal F_i -\theta_i (x) \mathcal B_i(x,u) )\, \dei \varphi \,dx.
\end{equation}
Assumption  \eqref{b1}  yields
 \begin{equation*}\label{due}
 |(1-\theta_i (x) )\mathcal B_i(x,u)|\leqslant  |b_i (x)- T_Mb_i (x)| |u|^{\frac {\bar p}{p'_i}} \qquad i=1,...,N
 \end{equation*}
and
  \begin{equation}\label{tre}
|\theta_i (x) \mathcal B_i(x,u) | \leqslant M |u|^{\frac{\bar p}{p'_i}}.
  \end{equation}

 Using \eqref{distlim} we can find a sufficiently large constant  $M>0$ independent on $i$, such that for every  $i=1,...,N$ the norm
$\|b_i - T_Mb_i\|_{L^{\frac{Np'_i}{\bar p},\infty } }$ is small in order to verify the smallness assumption of Step 1.

Hence,  we can apply  the result obtained in the previous step  to \eqref{solthetabis} whenever $
G_i(x,u):=[  \mathcal |F_i|^{p_i-2} \mathcal F_i(x) -\theta_i (x) \mathcal B_i(x,u)] \in L^{\frac {r_i}{p_i-1}}_{loc} (\Om),
$ for every $i=1,....,N.$

 Using \eqref{tre}, it is enough to have
\begin{equation}\label{finalboot}
|u|^{\frac{\bar p}{p_i}}\in L^{r_i}_{loc} (\Om), \mbox{ for every }i=1,....,N.
\end{equation}

At this point we observe that  if   $r_i\leq \frac {\overline{p}^*}{\overline p} p_i$ for every $i=1,....,N$, then \eqref{finalboot} holds directly by Sobolev inequality.  If otherwise there exist $i\in \{1,...,N\}$ such that $r_i>\frac {\overline{p}^*}{\overline p} p_i$ we use a bootstrap procedure  similar to the one used to conclude the proof of Step 1.

  Precisely, if there exist
  $i_j \in \{1,...,N\}$ such that
$r_{i_j}> \frac {\overline{p}^*}{\overline p} p_{i_{j}}$ , $j=1,...,m$, with $m\leq N$,  we replace  $r_{i_j} $, $j=1,...,m$  with $\frac {\overline{p}^*}{\overline p} p_{i_j}$ and  $\mu $  with $\mu_1:=\min\{\mu,\frac {\overline{p}^*}{\overline p}\}$ respectively. Now, if $\mu_1=\mu$ the proof is completed, since in this case we find $u\in L^{s(\mu)}_{loc}(\Om) $ with $s(\mu) $ as in \eqref{def s}. Otherwise, if $\mu_1= \frac {\overline{p}^*}{\overline p} $  we find  $u\in L^{s(\mu_1)}_{loc}(\Om)$ with $s (\mu_1)= \bar p^{**}$
and $s(\mu_1)-s(1) > \bar p^* - p_{\max}$.


%
\noindent Now, if $r_{i_j}\leq \frac {\overline{p}^{**}}{\overline p} p_{i_j}$ for every $j=1,...,m$, then $u\in L^{s(\mu_1)}_{loc}(\Om)$ gives \eqref{finalboot}. Otherwise we repeat the procedure again.
%
It is clear that, in a finite number of times, we can conclude our proof. $\blacksquare$

\begin{remark} In the previous proof  the definition of $\mu$ in \eqref{mu} as a minimum obviously appears. Indeed, using H\"{o}lder inequality in \eqref{stima prod B_fin} with exponent
$\zeta$ and $\zeta'$, it is enough to require $p_i \zeta\leq r_i$, $i=1,...,N.$ Hence, the best choice of such $\zeta$ is  $\mu$.
\end{remark}

\begin{remark}\label{remark0}

We stress that assumption $\overline{p}^*>p_{\max}$ is essential for our technique.  Indeed, at the end of the  Step 1 of previous proof, in order to start with the bootstrap argument we need $u\in L^{r_i}_{loc}(\Om) $ with $r_i\leq \bar p^*$ for every $i=1,..,N$ that means $p_i<\bar p^*$, $i=1,...,N.$

\end{remark}

\begin{remark}\label{remark1}
If $\Omega$ is bounded open set with Lipschitz boundary and we consider homogeneous Dirichlet problems we can argue as in Theorem \ref{mainTHM} to obtain a regularity  of solutions without  restriction $p_{\max}< \overline{p}^*$. Precisely we have the following result.
Assume that \eqref{ii1}-\eqref{b1} are fulfilled, let  $1<\bar p<N$  and let $r_1,\cdots,r_N$ be such that \eqref{hp-ri} holds.
 There exists a positive constant  $d=d(\vec r,N, \alpha, \vec{p})$  such that if
\begin{equation*}
\max_i \left\{ \emph{dist}_{L^{\frac{N p'_i }{\bar p},\infty}(\Om)} (b_i, L^\infty(\Om)):\right\} <d
\end{equation*}
and  $u\in W_0^{1,\vec p}(\Om)$
is a weak solution to \eqref{diffusion}  with $\mathcal F_i\in L^{r_i} (\Om)$, then
\begin{equation}\label{rem1}
u\in L^{s}(\Om) \quad \text{with }s=\max\{(\mu \bar p)^*,\mu p_{\max}\}.
\end{equation}
{Indeed when the $\max\{(\mu \bar p)^*,\mu p_{\max}\}=(\mu \bar p)^*$  the proof of \eqref{rem1} runs as Theorem \ref{mainTHM} taking into account that we are managing not local solutions. Otherwise if  $\max\{(\mu \bar p)^*,\mu p_{\max}\}=\mu p_{\max}$ we can reason as in Step 1 of Theorem \ref{mainTHM} but instead of \eqref{punto finale} we obtain


 \begin{equation}\label{stima regolarita bis}
  \begin{split}
  \left\| \frac{\partial}{\partial x_j}\left(|T_Ku|^{\frac{\gamma+1}{p_{\max}}+1}\right)\right\|_{L^{p_{\max}}}^{p_{\max}}
  & \leq C \sum_{i=1}^N  \left( \int_{\Omega} |\mathcal F_i|^{p_i \mu } dx\right)^{\frac{1}{ \mu }}\left( \int_{\Omega} |T_Ku|^{(\gamma+1) \frac{\mu}{\mu-1}}  dx\right)^{    \left(1-\frac 1\mu \right)}.
  \end{split}
  \end{equation}
Now choosing $\gamma>0$ such that
\begin{equation}\label{gamma_bis}
 (\gamma+1)   \left( \frac{\mu}{\mu-1} \right) \frac{ p_{\max}}{\gamma+1+p_{\max}}=p_{\max},
  \end{equation}
by Poincar\'{e} inequality \eqref{dis poincare},   \eqref{stima regolarita bis} becomes
\begin{equation}\label{limite bis}
\begin{split}
 \left\| \frac{\partial}{\partial x_j}\left(|T_Ku|^{\frac{\gamma+1}{p_{\max}}+1}  \right)\right\|_{L^{p_{\max}}}
 ^{\frac{p_{\max}}{\mu}}&
 \leq C\sum_{i=1}^N \left\|\mathcal F_i\right\|_{L^{r_i}}^{\frac{r_i}{\mu}},
\end{split}
\end{equation}
where the constant $C$ depends now also by $\Om.$
Letting $K\rightarrow +\infty$ in \eqref{limite bis}, observing that  equality \eqref{gamma_bis} implies  $\left( \frac{\gamma+1}{p_{\max}}+1 \right)= \mu$, and using again  Poincar\'{e} inequality \eqref{dis poincare},  we obtain \eqref{rem1} with $s=\mu p_{\max}$.

}


\end{remark}

\begin{remark}   Let us consider the homogeneous Dirichlet problem in a bounded open set with Lipschitz boundary $\Omega$ under the assumptions \eqref{ii1}-\eqref{ii2}, when  \eqref{b1} is replaced by

\begin{equation}\label{b2}
|\mathcal B_i(x,s)| \leqslant b_i(x)|s|^{ p_i-1}\quad \forall i
\end{equation}
for a.e. $x\in \Om$ and for every $s\in \mathbb R$  and $b_i\in L^{\infty}(\Omega)$ for all $i$. In this case we obtain the same regularity result as in Remark \ref{remark1}. Indeed one can obtain the analogous  inequality of \eqref{punto finale} using Poincar\'{e} inequality \eqref{dis poincare} instead of Sobolev inequalities. Starting from the obtained estimate, one can conclude the proof as before under assumption \eqref{b2}.
\end{remark}

\section{Proof of Theorem \ref{Th2}}

The first step of our proof is the boundedness of a local weak solution  without the lower order terms in order to adapt Stampacchia's arguments \cite{S} to the anisotropic case in the same spirit of Lemma 5.4 in \cite{Lady} when one deals with local solutions. Finally, our assumptions on the summability of the coefficients $b_i$  allow us  to apply  Corollary 1.3 concluding also  when $\mathcal{B}_i\not\equiv0$.

\medskip

\textit{Step 1. Proof assuming that $\mathcal B_i$ vanishes.}
Using the same test function $w$ as in Theorem \ref{mainTHM} (Step 1)
and assumptions \eqref{ii1} and \eqref{ii2}, we get
\begin{align}\label{I0_L}
\alpha \sum_{i=1}^N\int_{\Omega_k}\!&\left|\frac{\partial}{\partial x_i}v\right|^{p_i}\!\!\varphi^q dx
\leq \sum_{i=1}^N\!\beta_i\!\!\int_{\Omega_k} \!\left|\frac{\partial}{\partial x_i}v\right|^{p_i-1}\!\!\varphi^{q-1}
\left|\frac{\partial}{\partial x_i}\varphi\right|v dx
\!+\!\sum_{i=1}^N\int_{\Omega_k}\!\! |\mathcal F_i|^{p_i-1}\left|\frac{\partial}{\partial x_i}v\right|\varphi^q  dx
\nonumber \\
&+\sum_{i=1}^N\int_{\Omega_k} |\mathcal F_i|^{p_i-1}v\left|\frac{\partial}{\partial x_i}\varphi\right|q\varphi^{q-1} dx:=L_1+L_2+L_3.
\end{align}
Since $L_1=I_1$ and $L_2+L_3=I_5+I_6$, where $I_1, I_5, I_6$ are defined in Theorem \ref{mainTHM} (Step 1), putting together inequalities \eqref{I1}, \eqref{I6} and choosing $\varepsilon$ small enough, inequality \eqref{I0_L} became
\begin{equation}\label{I_finale_L}
\begin{split}
 \sum_{i=1}^N\int_{\Omega_k} \left|\frac{\partial}{\partial x_i}v\right|^{p_i} \varphi^{q}dx
\leq &
C\left[
\sum_{i=1}^N\int_{\Omega_k} v^{p_i}\left|\frac{\partial}{\partial x_i}\varphi\right|^{p_i}\varphi^{q-p_i} dx
+ \sum_{i=1}^N\int_{\Omega_k} |\mathcal F_i|^{p_i}\varphi^q dx\right],
\end{split}
\end{equation}
for suitable positive constant $C=C(\alpha,N,\vec{p}, \vec{\beta})$ that will change from line to line.


Now we choose the cut-off function $\varphi$. Let $ \sigma>\tau>0$, we fix two concentric balls $B_{\tau}\subset B_{\sigma} \subset\subset \Omega$ of radii $\tau$ and $\sigma$ respectively. We consider $\varphi$ such that
$0\leq\varphi\leq 1, \varphi\equiv 1$ in $B_{\tau}$, $\varphi\equiv 0$ in $\Omega \setminus B_{\sigma}$ and $|\nabla\varphi|\leq \frac{2}{{\sigma-\tau}}$. We put $\Omega_{k,\sigma}=\Omega_k\cap B_{\sigma}$.
By \eqref{I_finale_L}, taking $q> \max_{i} p_i$ and using H\"{o}lder inequality we get
\begin{equation*}
\begin{split}
\sum_{i=1}^N\int_{\Omega_{k,\tau}} \left|\frac{\partial}{\partial x_i}v\right|^{p_i} dx
\leq &
C\sum_{i=1}^N\frac{1}{(\sigma-\tau)^{p_i}}\int_{\Omega_{k,\sigma}}  (|u|-k)^{p_i} dx+C\sum_{i=1}^N \|\mathcal F_i\|^{p_i}_{L^{\mu p_i}(\Omega_{k,\sigma})} |\Omega_{k,\sigma}|^{1-\frac{1}{\mu}}.
\end{split}
\end{equation*}
We stress that since $\mathcal F_i\in L^{r_i}_{loc}(\Omega)$, then $\mathcal F_i\in L_{loc}^{p_i\mu}(\Om)$ for every $i=1,...,N$.
Moreover there exists $\widehat{k}>0$ such that $|\Omega_{k,\sigma}|<1$ for $k\geq\widehat{k}$,  and then
$$
|\Omega_{k,\sigma}|^{1-\frac{1}{\mu}}
\leq
|\Omega_{k,\sigma}|^{1-\frac{\overline{p}}{N}}$$
under the assumption \eqref{hp-ri_boun}. Then, for $k>\widehat{k}$
\begin{equation}\label{STIMA LEMMA}
\!\!\sum_{i=1}^N\!\int_{\Omega_{k,\tau}}\!\! \left|\frac{\partial}{\partial x_i}v\right|^{p_i} \!dx
\!\leq\!
%
C\!\sum_{i=1}^N\!
\left[\!
{(\sigma-\tau)^{-p_i}}\!\!\int_{\Omega_{k,\sigma}}\!\!\! \!\!\! (|u|-k)^{p_i} dx
\!+\! \sigma^{-\delta N}|\Omega_{k,\sigma}|^{1-\frac{\overline{p}}{N}+\delta}
\|\mathcal F_i\|^{p_i}_{L^{\mu p_i}(\Omega_{k,\sigma})}
\right],
\end{equation}
where $\delta$ is a positive constant that we choose later. Inequality \eqref{STIMA LEMMA} is the key ingredient to conclude that we can choose $k_0>0$ such that
\begin{equation}\label{claim}
\int_{\Omega_{2k_0,\sigma-\sigma_0}} (u-2k_0)\, dx =0,
\end{equation}
for  a fixed $0<\sigma_0<\sigma$,
\textit{i.e.} $u$ is locally bounded when the norms of $b_i$ are small enough. The claim \eqref{claim} will be proved in the next step following the idea of Lemma 5.4 of \cite{Lady}.

\textit{Step 2.  Proof of \eqref{claim}.}

Let us consider the sequence of concentric balls $B_{\rho_h}$, with
$\rho_h=\sigma-\sigma_0
+\frac{\sigma_0}{2^h}$
and the sequence
$k_h=2k_0-\frac{k_0}{2^h}$,
where $k_0>0$ will be fixed later. We will denote for $h=0,1,...$
\begin{equation*}\label{Jh}
J_h^i=\int_{\Omega_{k_h,\rho_h}}(|u|-k_h)^{p_i}\, dx \quad \text{and}\quad \zeta_h(x)=\zeta\left(2^{h+1}(|x|-\sigma+\sigma_0)\right),
\end{equation*}
where $\zeta(s)$ is a continuously differentiable nonincreasing function on $\mathbb{R}$ that is equal to 1 for $s\leq\sigma_0$ and equal to $0$ for $s\geq\frac{3}{2}\sigma_0$.
We stress that $\zeta_n\equiv 1$ inside the ball $B_{\rho_{h+1}}$ and $\zeta_n\equiv 0$ outside the ball $B_{\overline{\rho}_h}$, where $\overline{\rho}_h=\frac{1}{2}(\rho_{h+1}+\rho_{h})$.
We have
$$J_{h+1}^i\leq \int_{\Omega_{k_{h+1},\overline{\rho}_h}}
(|u|-k_{h+1})^{p_i}\zeta_h^{p_i}\, dx.$$
Moreover by Poincar\'{e} inequality obtained by Sobolev and H\"{o}lder inequalities we get
\begin{equation}\label{5.14}
\begin{split}
J_{h+1}^i&\leq C
|\Omega_{k_{h+1},\overline{\rho}_h}|
^{1-\frac{p_i}{\overline{p}}+\frac{p_i}{N}} \left(\prod_{j=1}^N
\left\{\int_{\Omega_{k_{h+1},\overline{\rho}_h}}
\left(\frac{\partial}{\partial x_j}
[(|u|-k_{h+1})\zeta_h]\right)^{p_j}\right\}^{\frac{1}{p_jN}}\right)^{p_i}
\\
&\leq C
|\Omega_{k_{h+1},\overline{\rho}_h}|
^{1-\frac{p_i}{\overline{p}}+\frac{p_i}{N}} \left(\prod_{j=1}^N
\left[\int_{\Omega_{k_{h+1},\overline{\rho}_h}}
\left|\frac{\partial}{\partial x_j}u
\right|^{p_j} \, dx
+
\nu^j 2^{p_j h}J_h^j
\right]^{\frac{1}{p_jN}}\right)^{p_i},
\end{split}
\end{equation}
where $C=C(\alpha,N,\vec{p}, \vec{\beta})$ and
$$\nu^i=\max_{s\in [\sigma_0,\frac{3}{2} \sigma_0]}[\zeta'(s)]^{p_i}.$$
Now we write \eqref{STIMA LEMMA} with $k=k_{h+1}$ and $\sigma=\rho_h$ and $\tau=\overline{\rho}_h$:

\begin{align}\label{STIMA LEMMA 2}
&\sum_{i=1}^N\int_{\Omega_{k_{h+1},\overline{\rho}_h}} \left|\frac{\partial}{\partial x_i}v\right|^{p_i} dx\\
&\leq
%
C\left[\sum_{i=1}^N\frac{1}{(\rho_h-\overline{\rho}_h)^{p_i}}
\int_{\Omega_{k_{h+1},\rho_h}}  (|u|-k_{h+1})^{p_i} dx+
\rho_h^{-\delta N}|\Omega_{k_{h+1},\rho_h}|^{1-\frac{\overline{p}}{N}+\delta}\sum_{i=1}^N \|\mathcal F_i\|^{p_i}_{L^{\mu p_i}}\right]
\nonumber \\
&\leq C\left(  \rho_h^{-\delta N}|\Omega_{k_{h+1},\rho_h}|^{1-\frac{\overline{p}}{N}+\delta}+\sum_{i=1}^N
2^{(h+3)p_i}
J_h^i\right),\nonumber
\end{align}
where $C=C(\alpha, \|\mathcal F_i\|_{L^{\mu p_i}},N,\vec{p}, \vec{\beta})$. Let us find a bound for the measure of the set $\Omega_{k_{h+1},\rho_h}$. We have
\begin{equation}\label{S1}
J^i_h\geq\int_{\Omega_{k_{h+1},\rho_h}}(|u|-k_h)^{p_i}\geq
(k_{h+1}-k_{h})^{p_i}|\Omega_{k_{h+1},\rho_h}|=2^{-(h+1)p_i}k_0^{p_i}|\Omega_{k_{h+1},\rho_h}|.
\end{equation}

Putting together \eqref{S1}, \eqref{STIMA LEMMA 2}, inequality \eqref{5.14} can be rewrite as
\begin{equation*}
\begin{split}
J_{h+1}^i&\leq C
|\Omega_{k_{h+1},\overline{\rho}_h}|
^{1-\frac{p_i}{\overline{p}}+\frac{p_i}{N}}
\left\{
\prod_{j=1}^N\left[\sum_{m=1}^N\int_{\Omega_{k_{h+1},\overline{\rho}_h}}
\left|\frac{\partial}{\partial x_m}u
\right|^{p_m} \, dx
+
\sum_{m=1}^N\nu^m 2^{p_m h}J_h^m
\right]^{\frac{1}{p_jN}}
\right\}^{p_i}
\\
&
\leq C
\left(2^{(h+1)p_i}k_0^{-p_i}J_h^i\right)
^{1-\frac{p_i}{\overline{p}}+\frac{p_i}{N}}
\left[\sum_{m=1}^N\int_{\Omega_{k_{h+1},\overline{\rho}_h}}
\left|\frac{\partial}{\partial x_m}u
\right|^{p_m} \, dx
+
\sum_{m=1}^N\nu^m 2^{p_m h}J_h^m
\right]^{\frac{p_i}{\overline{p}}}
\\
& \leq C
\left(2^{(h+1)p_i}k_0^{-p_i}J_h^i\right)
^{1-\frac{p_i}{\overline{p}}+\frac{p_i}{N}}
\left[\sum_{m=1}^N\left(
2^{(h+3)p_m}
J_h^m
+ \nu^m 2^{p_m h}J_h^m\right)+
|\Omega_{k_{h+1},\rho_h}|^{1-\frac{\overline{p}}{N}+\delta}
\right]^{\frac{p_i}{\overline{p}}}
\\
& \leq C
\left(2^{(h+1)p_i}k_0^{-p_i}J_h^i\right)
^{1-\frac{p_i}{\overline{p}}+\frac{p_i}{N}}
\left[
\sum_{m=1}^N \left(
2^{(h+3)p_m}
J_h^m
+\nu^m 2^{p_m h}J_h^m\right)+ \left(2^{(h+1)p_i}k_0^{-p_i}J_h^i\right)^{1-\frac{\overline{p}}{N}+\delta}
\right]^{\frac{p_i}{\overline{p}}}
\\
& \leq C 2^{hp_{\max}}
k_0^{-p_{\min}+\frac{p^2_{\max}}{\overline{p}^*}}
\left[
 2^{h \frac{{p_{\max}}^2}{\overline{p}}}
(J_h^i)^{1+\frac{p_i}{N}-\frac{p_i}{\overline{p}}}
\left(\sum_{m=1}^N
J_h^m\right)^{\frac{p_i}{\overline{p}}}
+ 2^{h(1+\delta)\frac{{p_{\max}}^2}{\overline{p}}}k_0^{\frac{p_{\max}^2}{N}-\frac{p_{\min}^2}{\overline{p}}(1+\delta)}(J_h^i)^{1+\delta\frac{p_i}{\overline{p}}}\right.
\\
&+\left.
2^{h\frac{{p_{\max}}^2}{\overline{p}}}(\max_{m}\nu^m)^{\frac{p_i}{\overline{p}}}
(J_h^i)^{1+\frac{p_i}{N}-\frac{p_i}{\overline{p}}}
\left(\sum_{m=1}^N J_h^m\right)^{\frac{p_i}{\overline{p}}}
\right],
\end{split}
\end{equation*}
where $p_{\min}= \min_{i}\{p_1,...,p_N\}$. Since $\overline{p}^{*}>p_{\max}$ the exponent $1+\frac{p_i}{N}-\frac{p_i}{\overline{p}}>0$
and setting $Y_{h}=\sum_{m=1}^N J_h^m$, one can rewrite the previous inequality as follows
\begin{align}\label{5.16}
J_{h+1}^i\leq&
C 2^{hp_{\max}}
k_0^{-p_{\min}+\frac{p^2_{\max}}{\overline{p}^*}}
(Y_h)^{1+\delta\frac{p_i}{\overline{p}}}
\left[
2^{h \frac{{p_{\max}}^2}{\overline{p}}}
\left(Y_h\right)^{{\frac{p_i}{N}-\delta\frac{p_i}{\overline{p}}}}+
 2^{h(1+\delta)\frac{{p_{\max}}^2}{\overline{p}}}k_0^{\frac{p^2_{\max}}{N}-\frac{p^2_{\min}}{\overline{p}}(1+\delta)}
\right.
\nonumber \\
&+\left.
2^{h\frac{{p_{\max}}^2}{\overline{p}}}\max_{i}((\max_{m}\nu^m)^{\frac{p_i}{\overline{p}}})
\left(Y_h\right)^{{\frac{p_i}{N}-\delta\frac{p_i}{\overline{p}}}}
\right]
\\
\leq&
C 2^{h(p_{\max}+\frac{p_{\max}^2}{\overline{p}})}
k_0^{-p_{\min}+\frac{p_{\max}}{\overline{p}^*}}
(Y_h)^{1+\delta\frac{p_i}{\overline{p}}}
\left[ (Y_0)^{\frac{p_i}{N}-\delta\frac{p_i}{\overline{p}}}+1
\right],\nonumber
\end{align}
where the last inequality follows taking $\delta<\frac{\overline{p}}{N}$ and observing that $J_h^m$ are decreasing. Putting $\delta'=\frac{\delta p_i}{\overline{p}}$ and summarizing left and right side of \eqref{5.16}, we get
$$
Y_{h+1}\leq
NC 2^{h(p_{\max}+\frac{p_{\max}^2}{\overline{p}})}
k_0^{-p_{\min}+\frac{p^2_{\max}}{\overline{p}^*}}
(Y_h)^{1+\delta'}
\max_i\left[ (Y_0)^{\frac{p_i}{N}-\delta'}+1
\right].$$
Denoting $C_1=NC \max_i\left[ (Y_0)^{\frac{p_i}{N}-\delta'}+1
\right]$, $\omega=p_{\min}-\frac{p^2_{\max}}{\overline{p}^*}$ and $b= 2^{(p_{\max}+\frac{p_{\max}^2}{\overline{p}})}$ the previous inequality became
$$Y_{h+1}\leq
C_1 b^h
k_0^{-\omega}
(Y_h)^{1+\delta'}.$$
Choosing $k_0$ such that
$k_0=\max\{\widehat{k},1,C_1^{1/\omega}b^{1/[\omega\delta'(1+\delta')]}a^{\delta'/\omega}\}$,
we obtain
$$Y_1\leq C_1 k_0^{-\omega}Y_0^{1+\delta'}\leq k_0^{\omega/\delta'}C_1^{-1/\delta'}b^{-1/(\delta')^2}.$$
Now we are in position to apply Lemma 4.7 of \cite{Lady} in order to obtain \eqref{claim}.

\textit{Step 3. Dropping the assumption  that $\mathcal B_i\equiv  0$.}

In the general case where $b_i\in L^{\frac{r_i}{p_i-1}}_{loc}(\Om)$, with $r_1,\cdots,r_N$ satisfying \eqref{hp-ri_boun},
%
 we  rewrite equation \eqref{sol1} as

  \begin{equation}\label{solthetabis2}
\Si \int_\Om   \left (\mathcal{A}_i(x,\nabla u) \right) \dei \varphi \, dx = \Si \int_\Om  ( \mathcal |F_i|^{p_i-2} \mathcal F_i - \mathcal B_i(x,u) )\, \dei \varphi \,dx.
\end{equation}

Then, we can apply the result obtained in the previous steps to  \eqref{solthetabis2}. In fact, by \eqref{hp-ri_boun} and  by Theorem \ref{mainTHM},
we have that  $u\in L^q_{loc}(\Om)$ for every $q <+\infty$. Hence we can find  $p_i<s_i<r_i$ such that $\min_{i}\left\{ \frac{s_i}{p_i}\right\}>\frac{N}{\bar p}$
and
\begin{equation*}\label{finalboot2}
G_i(x,u):=[  |\mathcal F_i|^{p_i-2} \mathcal F_i(x) - \mathcal B_i(x,u)] \in L^{\frac {s_i}{p_i-1}}_{loc} (\Om),
\mbox{ for every } i=1,....,N.
\end{equation*}
$\blacksquare$

\begin{remark}\label{remarkFINE}
If $\Omega$ is a bounded open set with Lipschitz boundary and we consider homogeneous Dirichlet problems we can argue as in Theorem \ref{Th2} to obtain  the boundedness of solutions. Precisely, instead of \eqref{I_finale_L}, we obtain
\begin{equation*}
\begin{split}
\sum_{i=1}^N\int_{\Omega_k} \left|\frac{\partial}{\partial x_i}v\right|^{p_i} dx
\leq C \sum_{i=1}^N\int_{\Omega_k} |\mathcal F_i|^{p_i} dx,
\end{split}
\end{equation*}
where $C=C(\alpha, \vec p, \vec \beta,N) >0$. At this point, we can proceed
as in Theorem 2 in \cite{Str} and one can conclude the proof using the Stampacchia's Lemma (see \cite{S}) instead of Step 2 of Theorem \ref{Th2}.
\end{remark}
\section{Appendix}

In this appendix we prove a technical lemma (see also \cite{FVV}).

\begin{lemma}\label{L som}
Let $R>0,\alpha\geq0, \beta\geq0$ and $\theta_i>0$ for all $i=1,\cdots,N$ and $\Omega=\{\sum_{i=1}^N|x_i|^{\theta_i}<R\}$. If
\begin{equation}\label{J1}
\sum_{i=1}^N\frac{1}{\theta_i}+\frac{\beta}{\theta_j}>\alpha,
\end{equation}
then $\left(\sum_{i=1}^N|x_i|^{\theta_i}\right)^{-\alpha}|x_j|^\beta\in L^1(\Omega)$.
Otherwise
$\left(\sum_{i=1}^N|x_i|^{\theta_i}\right)^{-\alpha} \not \in L^1(\Omega)$.
\end{lemma}
\begin{proof}
Putting $\theta=\max \theta_i$ and $x_i=|y_i|^{\frac{\theta}{\theta_i}}\text{sign} y_i$ we have
\begin{equation*}
\begin{split}
\int_{\{\sum_{i=1}^N|x_i|^{\theta_i}<R\}}\left(\sum_{i=1}^N|x_i|^{\theta_i}\right)^{-\alpha}|x_j|^\beta\, dx&=
\int_{\{\sum_{i=1}^N|y_i|^{\theta}<R\}}\left(\sum_{i=1}^N|y_i|^{\theta}\right)^{-\alpha}
|y_j|^{\frac{\theta}{\theta_j}\beta}
\prod_{i=1}^N\frac{\theta}{\theta_i}|y_i|^{\frac{\theta}{\theta_i}-1}
\, dy\\
&\leq
C_1\int_{\{\sum_{i=1}^N|y_i|^{\theta}<R\}}|y|^{-\theta\alpha+\frac{\theta}{\theta_j}\beta+\sum_{i=1}^N\frac{\theta}{\theta_i}-N}
\, dy\\
&\leq
C_1\int_{\{|y|^{\theta}<C_2R\}}|y|^{-\theta\alpha+\frac{\theta}{\theta_j}\beta+\sum_{i=1}^N\frac{\theta}{\theta_i}-N}
\, dy
\end{split}
\end{equation*}
for suitable positive constants $C_1,C_2$. The last integral is finite if \eqref{J1} holds. Otherwise putting $\theta=\min \theta_i$ and $x_i=|y_i|^{\frac{\theta}{\theta_i}}\text{sign} y_i$ we have
\begin{equation*}
\begin{split}
\int_{\{\sum_{i=1}^N|x_i|^{\theta_i}<R\}}\left(\sum_{i=1}^N|x_i|^{\theta_i}\right)^{-\alpha}|x_j|^{\beta}\, dx&=
\int_{\{\sum_{i=1}^N|y_i|^{\theta}<R\}}\left(\sum_{i=1}^N|y_i|^{\theta}\right)^{-\alpha}|y_j|^{\frac{\theta}{\theta_j}\beta}
\prod_{i=1}^N\frac{\theta}{\theta_i}|y_i|^{\frac{\theta}{\theta_i}-1}
\, dy
\\
\geq
&C_3 \int_{\{|y|^{\theta}<R/N\}}|y|^{-\theta\alpha+\sum_{i=1}^N\frac{\theta}{\theta_i}-N}
|y_j|^{\frac{\theta}{\theta_j}\beta}\, dy
\end{split}
\end{equation*}
where $ C_3$ is a positive constant. Now we use spherical coordinates, in which the coordinates consist of a radial coordinate $\rho$ and $N-1$ angular coordinates $\psi_1,\psi_2,\cdots,\psi_{N-1}$, where the angles $\psi_1,\psi_2,\cdots,\psi_{N-2}$ range over $[0,\pi]$ and $\psi_{N-1}$ ranges over $[0,2\pi)$. For example $y_1$ coordinate becomes $y_1= \rho \cos \psi_1$ and then
\begin{equation*}
\begin{split}
\int_{\{|y|^{\theta}<R/N\}}&|y|^{-\theta\alpha+\sum_{i=1}^N\frac{\theta}{\theta_i}-N}
|y_1|^{\frac{\theta}{\theta_1}\beta}\, dy
\\
& =\int_0^{2\pi}\int_{0}^{\frac{R}{N}} \rho^{-\theta\alpha+\sum_{i=1}^N\frac{\theta}{\theta_i}-N}|\rho \cos \psi_1|^{\frac{\theta}{\theta_1}\beta}
\rho^{N-1}    \prod_{i=1}^{N-2} \sin ^{N-1-i}\psi_{i} \,d\rho\,d\psi_{1}\cdots d\psi_{N-1}\\
&
=C(N) \int_{0}^{\frac{R}{N}} \rho^{-\theta\alpha+\sum_{i=1}^N\frac{\theta}{\theta_i}-N+\frac{\theta}{\theta_1}\beta +N-1}\,d\rho.
\end{split}
\end{equation*}
The last integral is finite if 
\eqref{J1} does not hold and for $j=2,\cdots,N$ the proof runs similarly.
\end{proof}

\section*{Acknowledgements}

The authors are members of Gruppo Nazionale per l'Analisi Matematica, la Probabilit\`a e le loro Applicazioni (GNAMPA) of the Istituto Nazionale di  Alta Matematica (INdAM). 
Research partially supported by project    Vain-Hopes within the program VALERE: VAnviteLli  pEr la RicErca.

\end{document}